\documentclass{birkjour_t2}

\usepackage{graphicx}
\usepackage{amsmath}
\usepackage{amssymb}
\usepackage{epstopdf}
\usepackage{xcolor}

\usepackage[pagewise]{lineno}\linenumbers
\usepackage{enumerate}
 \usepackage{amsaddr}
\usepackage{array}
\usepackage{ctable}

\usepackage{caption}
\usepackage{subcaption}

\newcommand{\eps}{\varepsilon}
\newcommand{\R}{\mathbb{R}}
\newcommand{\Z}{\mathbb{Z}}
\newcommand{\C}{\mathbb{C}}
\newcommand{\N}{\mathbb{N}}

\newcommand{\rmi}{\mathrm{i}}
\newcommand{\rme}{\mathrm{e}}
\newcommand{\rmO}{\mathrm{O}}

\newcommand{\rmo}{\mathrm{o}}

\newtheorem{Lemma}{Lemma}[section]
\newtheorem*{Lemma*}{Lemma}
\newtheorem{Theorem}{Theorem}
\newtheorem*{Theorem*}{Theorem}
\newtheorem{Proposition}[Lemma]{Proposition}
\newtheorem*{Proposition*}{Proposition}

\newtheorem{Remark}[Lemma]{Remark}

\newtheorem{Hypothesis}[Lemma]{Hypothesis}
\newtheorem*{Hypothesis*}{Hypothesis}

\title{
Can large inhomogeneities generate target patterns?}

\thanks{ORCID: 0000-0002-7724-3794 \\
This work is supported by NSF DMS-1911742.}

\author{Gabriela Jaramillo}
\email{gabriela@math.uh.edu}

\address{University of Houston, Department of Mathematics,\\
 3551 Cullen Blvd., 
 Houston, TX 77204-3008, USA\\}

\begin{document}
'\nolinenumbers
\begin{abstract}
We study the existence of target patterns in oscillatory media with weak local coupling
 and in the presence of an impurity, or defect.
We model these systems using a viscous eikonal equation posed on the plane, and represent
the defect as a perturbation. In contrast to previous results we consider large defects, 
which we describe using a  function with slow algebraic decay, 
i.e. $g \sim \rmO(1/|x|^m)$ for $m \in (1,2]$.
We prove that these defects are able to generate target patterns
 and that, just as in the case of strongly localized impurities, 
 their frequency  is small beyond all orders of the small parameter describing their strength.
Our analysis consists of finding two approximations to target pattern solutions, one which  
is valid at intermediate scales and a second one which is valid in the far field. 
This is done using weighted Sobolev spaces, which allow us to recover Fredholm properties of the
relevant linear operators, as well as the implicit function theorem, which is then used to prove existence. 
By matching the intermediate and far field approximations we then determine the 
 frequency of the pattern that is selected by the system.

\end{abstract}

\maketitle

\vspace*{0.2in}

{\small



}

\vspace*{0.2in}

\section{Introduction}

Target patterns are coherent structures that emerge in excitable and in oscillatory media.
They are characterized by concentric waves that expand away from a center, or core region,
 creating a `bull's-eye' pattern.
 Although often associated with the  Belousov-Zhabotinsky reaction \cite{zaikin1970}, 
 they also appear
in colonies of slime mold \cite{alcantara1974, durston1974}, in the oxidation of carbon monoxide on platinum \cite{stich2004}, and in brain tissue \cite{townsend2015}.

In this paper we will  focus on target patterns that arise in oscillatory media,
where three key mechanisms, or processes, contribute to their formation.
The first mechanism is associated with the intrinsic dynamics of the system, which must support a limit cycle 
that results in uniform time oscillations. The second 
is a transport process that allows for different spatial regions to interact,
such as diffusion in chemical reactions, or coupling between neurons in brain tissue.
While these two processes are enough to generate traveling and spiral waves,
to obtain target patterns one needs a third ingredient, a defect.
Indeed, it is believed that the role of defects, or impurities, is to
alter the dynamics of the system in a localized area
resulting in a change in the frequency of the time oscillations. 
As a consequence, these defects act as pacemakers entraining the rest of the medium and forming 
target patterns.

While experiments and previous analytical results  confirm
 that small localized defects give rise to these patterns, 
\cite{tyson1980, kopell1981target, hagan1981, kuramoto2003, hendrey2000, stich2002, stich2004, kollar2007, jaramillo2016, jaramillo2018},
in this paper we want to determine the exact level of localization that is needed to generate them.
In particular, assuming the inhomogeneity is modeled as a function with 
algebraic decay of order $\rmO(1/|x|^m)$, we want to determine how small we can take $m>0$ and still obtain
a well defined target pattern.

To simplify the analysis we concentrate only on systems which involve weak local coupling.
Because it is well known that under this assumption
 the amplitude of oscillations is tied, or enslaved, to the dynamics of the phase,
this allows us to focus our analysis on this  last variable.
Indeed, the results presented in \cite{doelman2009} show that 
 coherent structures in these systems are well described by
the following viscous eikonal equation 
\begin{equation}\label{e:viscous}
 \phi_t = \Delta \phi - b |\nabla \phi |^2 - \eps g(x), \quad x \in \R^2,
 \end{equation}
where the perturbation, $\eps g$, represents the defect.
 The above expression is derived using a multiple scale analysis and
 it therefore models phase changes that occur over long spatial and time scales.
 In this context, target patterns then correspond to solutions of the form 
 $\phi(x,t) = \tilde{\phi}(x) - \Omega t$, satisfying the boundary condition $\nabla \phi \to k$ 
 as $|x| \to \infty$, where the constant $k$ then represents the pattern's wavenumber.

Our motivation for considering large inhomogeneities is three fold.
First, in all previous work
the level of localization imposed on the inhomogeneity was tied to the tools used 
to prove the existence of these patterns. 
Yet, numerical simulation 
 like the ones presented here in Section \ref{s:simulations},
show that these assumptions can be relaxed. 
For example, in \cite{stich2006} 
defects are modeled as functions with compact support
and target pattern solutions are found using separation of variables.
In contrast, in \cite{kollar2007} the authors use spatial dynamics to prove the existence of these patterns.
This then allows them to model the impurities as radially symmetric functions with exponential decay.
 In \cite{jaramillo2018}, thanks to the use of weighted Sobolev
spaces, this assumption is relaxed
and general (non-radially symmetric) inhomogeneities 
with  decay of order $\rmO(1/|x|^{m})$, $m>2$, are considered.

Although using different approaches,  the references mentioned above show that 
target patterns can only be generated by inhomogeneities with 
a postive and finite mass $M = \int_{\R^2} g$.
This obviously restricts  the level of decay of $g$ to be of order $\rmo(1/|x|^2)$.
However, our numerical simulations show that  one 
can obtain target patterns even in the case when the defect is assumed to decay
only at order $\rmO(1/|x|^{m})$, for $m \in (1,2]$.
We are therefore interested in proving the existence of target patterns
for these `large' inhomogeneities of infinite mass.

Our  second reason for considering this problem
is tied to the existence of spiral waves in oscillatory media with nonlocal coupling.
In \cite{jaramillo2022} it was shown that  the dynamics of these patterns are well described by the following amplitude equation
\[ 0 = K \ast w + \lambda w + \alpha |w|^2w + \rmO(\eps), \quad w \in \C, \qquad 0< \eps <<1\]
where $w$ is a radially symmetric complex-valued function, and
$K$ is a symmetric convolution kernel of diffusive type.
Additional assumptions on $K$ imply that formally one can write this operator
 as $(1- \eps  \Delta_1)^{-1}\sigma \Delta_1$, and suggest
preconditioning the above equation with $(1-\eps  \Delta_1)$, where $\Delta_1 = \partial_{rr} + \frac{1}{r} \partial_r - \frac{1}{r^2}$.
This then results in the following expression, which perhaps not surprisingly resembles the complex Ginzburg-Landau equation,
\[ 0 = \beta \Delta_1 w + \lambda w + \alpha |w|^2w + \rmO(\eps), \quad \beta = ( \sigma - \eps \lambda), \quad r \in [0,\infty). \]

From there, a similar multiple-scale analysis as the one carried out  in \cite{doelman2009} and that we also summarize in  Appendix, gives a hierarchy of equations at different powers of a small parameter $\delta= \delta(\eps)$.
 In particular, at order $\delta^2$ one finds  the steady state viscous eikonal equation, 
\begin{equation}\label{e:steadyeik}
 -\Omega = \Delta_1 \phi - b | \nabla \phi|^2 - \delta^2 g,
 \end{equation}
as a description of  the phase dynamics of spiral waves.
However, in contrast to the case of target patterns, here the inhomogeneity does not represent a defect, 
but is instead related to the small variations, $\rho$, of the amplitude of the pattern.
More precisely, $g \sim ( 1- \rho^2)$. 
Although not immediately obvious, one can check using numerical simulations
that the perturbation $(1- \rho^2)$ decays at infinity at order $\rmO(1/|x|^2)$ (see the Appendix A).
Therefore, the particular viscous eikonal equation that is connected to the phase dynamics
of spiral waves in these systems
is the same equation that we are trying to solve.

Finally, our third motivation comes from the following change of variables, $\phi= -\frac{1}{b} \log(\Psi)$, 
which transforms the steady state viscous eikonal equation, \eqref{e:steadyeik}, into
a Schr\"odinger eigenvalue problem with potential $\eps g$,
\[ \Omega \Psi = \Delta \Psi + \eps g(x) \Psi. \]
The transformation also shows that our target pattern solutions correspond to bound states of this 
operator. The only result solving the above eigenvalue problem
 that we are aware of is that of Simon \cite{simon1979},
who proved that in the two dimensional case and
under the assumption of localized potentials, i.e. $\int_{\R^2} g(x) (1+x^2) \;dx< \infty$,
 bound states exists if and only if the mass  $\int_{\R^2} g(x) \;dx > 0$.

Notice that in the context of the Schr\"odinger operator, 
our problem corresponds to the `supercritical' case,
in the sense that the potential, $g$, no longer corresponds to a bounded perturbation of 
the Laplacian. To see this, fix $g (x) = 1/(1+ |x|)^m$ with $1< m \leq 2$, and consider the rescaling $ y = \gamma x$. 
The Schr\"odinger operator then reads 
$ \gamma^2 \Delta_y \Psi + \eps \frac{\gamma^m}{( \gamma+  |y|)^m} \Psi$,
and it is then clear that if we choose $\gamma$ small, the potential is actually `large' in the far field.
Consequently, the results from \cite{simon1979} no  longer apply for the case considered here.

In this paper we show that target pattern solutions
to the viscous eikonal equation, or equivalently, bound states to the above Schr\"odinger operator, exists
even for these large inhomogeneities.
As with small defects, we prove that target pattern solutions have frequencies, $\Omega$, that 
are small beyond all orders of the parameter $\eps$. 
Consequently one cannot use a regular perturbation expansion to justify existence.
To resolve this issue we first find two approximations to target patterns,
 one  which is valid at intermediate scales and second one that accounts for the far field behavior of the solution.
By matching these two approximations we are then  able to
determine the unique value of the frequency selected by the system.

 It is in the course of this analysis that one sees that
the slow decay rate of the inhomogeneity plays a major role in shaping 
the solution at intermediate scales.
This is the main difference between the analysis presented here and that  of \cite{jaramillo2018}, 
where inhomogeneities of finite mass are considered.
It is also why we will split defects into a core region and a far field region,
reflecting the fact that 
the defects we work with are still too small to alter the shape of the pattern at large scales,
but do contribute to the form of the equation at intermediate scales.
In particular, we write the impurity as the sum two functions, 
defined as
\begin{equation}\label{e:g}
 g_c = (1-\chi^D) g \qquad g_f= \chi^D g,
 \end{equation}
where $\chi^D$ is a $C^{\infty}$ radial cut-off function, with 
$\chi^D(|x|) =0$ for $|x|<D$ and $\chi^D(|x|) =1$ for $|x|>2D$. 
To prove the existence of target patterns, the value of the parameter
$D$ can remain arbitrary, so long as it is a finite number. 
This follows because even though in the above definition 
the function $g_c$ has compact support, 
our results hold for more general `core' functions. The only requirement being
that this core defect has finite mass.
We therefore make the following assumption.

\begin{Hypothesis}\label{h:g}
The inhomogeneity, $g$, lives in $H_{ \sigma} ^k(\R^2)$, with $k\geq2$ and $\sigma \in (0,1)$,
is radially symmetric, and positive. In addition, the defect can be split into the sum of  two positive functions, $g_c, g_f$, satisfying
\begin{itemize}
\setlength \itemsep{2ex}
\item The function $g_f$ is in $ H^k_\sigma(\R^2)$ for $0<\sigma<1$. In particular,
 $g_f \sim \rmO(1/r^m)$ as $r \to \infty$, with $1< m \leq 2$, while near the origin  $g_f(|x|) =0$ for $|x|<1$.
\item The function $g_c$ is in $H^k_\gamma(\R^2)$ for $\gamma>1$. In particular,
$g_c \sim \rmO(1/r^{d})$ with $d>2$ as $r \to \infty$.
\end{itemize}
\end{Hypothesis}

\begin{Remark}
The spaces $H^k_\sigma(\R^2)$, with $\sigma \in \R$, are weighted Sobolev  spaces
with norm
\[ \|u\|_{H^{k}_\sigma(\R^2)} = \sum_{|\alpha|\leq s} \| ( 1+ |x|^2)^{\sigma/2} D^\alpha u(x)\|_{L^2(\R^2)}.\]
Notice that for positive values of $\sigma$, they impose a level of  decay on functions.
For a precise definition of these spaces see Section \ref{s:preliminaries}.
\end{Remark}

With the above hypothesis and the approach just described, we prove the following result.

\begin{Theorem}\label{t:existence}
Let $k \geq 2$ and $\sigma \in (0, 1)$ and consider a function $g \in H^k_{r,\sigma}(\R^2)$ satisfying Hypothesis \ref{h:g}. Then, there exists a constant $\eps_0>0$ and a $C^1 ([0,\eps_0))$  family of eigenfunctions $\phi = \phi(r; \eps)$ and eigenvalues $\Omega= \Omega(\eps)$ that bifurcate from zero and solve
the equation
\[\Delta_0 \phi - b(\partial_r \phi) ^2 - \eps g(r) + \Omega =0 \qquad r= |x| \in [0,\infty).\]
Moreover, this family has the form
\[ 
\phi(r;\eps) = -\frac{1}{b} \chi^1(\Lambda r) \log( K_0(\Lambda r)) + \phi_1(r; \eps ) + \eps c , \qquad \Lambda^2 = b \Omega(\eps) \]
where
\begin{enumerate}[i)]
\item $c$ is a constant that depends on the initial conditions of the problem,
\item $K_0(z)$ represents the zeroth-order Modified Bessel function of the second kind,
\item $ \partial_r \phi_1 \in H^{k}_{r,\sigma}(\R^2) $, and
\item $\Omega= \Omega(\eps) = C(\eps) 4 e^{-2 \gamma_\eps}\exp[ 2/ a ]$, with
\[a = -\eps b \int_0^\infty g_c(r) \;r \;dr,\]
 and $C(\eps)$ a $C^1$ constant that depends on $\eps$.
\end{enumerate}
\end{Theorem}

\begin{Remark}

Notice that under  Hypothesis \ref{h:g}
the viscous eikonal equation, \eqref{e:viscous}, is invariant under rotations. 
As a result we can 
look for solutions that are radially symmetric. This assumption is made mainly for convenience, 
and one can follow the steps in \cite{jaramillo2018} to tackle the more general case of non-symmetric
inhomogeneties. 
\end{Remark}

\begin{Remark}\label{r:impurity}
If the inhomogeneity $g$ has strong algebraic decay, i.e. $g(r) \sim 1/r^m$ with $m>2$, then
we are back in the regime considered in \cite{jaramillo2018}.
In this case, the impurity 
has finite mass and there is no need to split this function into the sum
of its  core and far field functions. In fact, one can set $g =g_c = g_f$ and the 
above theorem is equivalent to Theorem 1 in \cite{jaramillo2018} with $ a = - \eps b \int_0^\infty g(r) r\;dr < \infty$.
\end{Remark}

\begin{Remark}\label{r:D}
While the exact form of the cut-off function $\chi^D$ appearing in the definition of $g_c$
 is not important for the proof of existence,
it does play a role when approximating the pattern's frequency, $\Omega$.
As our numerical simulations show, there is an optimal way of picking the parameter $D$
that allows one to obtain better estimates for the frequency,
 see Section \ref{s:simulations}.
If  a non-optimal choice is made,
 one can improve the estimates for $\Omega$ by using higher order approximations for the intermediate and far field solutions when carrying out the matched asymptotics, see Section \ref{s:existence}.

\end{Remark}

We close this section with some comments regarding the mathematical tools used in this paper.
As in reference \cite{jaramillo2018}, the proof of existence of solutions is based on the implicit function theorem.
This requires that the linearization about our first order approximation, $\phi= \phi_0$, be an invertible, or at least Fredholm
operator with closed range and finite dimensional kernel and cokernel. However, because the 
 equations are posed on the plane, we obtain linear operators that are second order differential operator 
 with essential  spectrum near the origin. In addition, the translational symmetry of the system implies that
these maps have a zero eigenvalue at the origin.
Consequently,  these operators are not invertible and they do not have a closed range when posed as 
maps between standard Sobolev spaces. To overcome this difficulty and
recover  Fredholm properties for these maps, we work instead with weighted Sobolev spaces.
In particular, we make use of the results from \cite{mcowen1979}, 
where  Fredholm properties for the Laplace operator are derived.
For other instances where this approach is used to prove existence of patterns  see
references  \cite{jaramillo2015, jaramillo2016, jaramillo2018, jaramillo2019}.

\subsection{Outline:}
The paper is organized as follows.
In Section \ref{s:preliminaries} we introduce a special class of weighted Sobolev spaces and summarize
 Fredholm properties for the Laplacian and related operators.
 In Section \ref{s:inter} we work with our model \eqref{e:viscous}
  and derive from it an equation that is valid at intermediate scales. 
  We then prove existence of solutions to this equation that are bounded near the origin
and that have appropriate growth conditions.
Next, in Section \ref{s:far} we work with the full model  \eqref{e:viscous} and,
treating the frequency as a parameter,
 find a first order approximation to target pattern solutions. 
Then, in Subsection \ref{ss:match} we  use matched asymptotics  to determine the value of the frequency, $\Omega$,
selected by the system. More precisely, we show that $\Omega$ is a $C^1$ function of the parameter $\eps$.
This then allows us to prove existence of solutions using the implicit function theorem.
 The analysis is complemented by numerical simulations presented in Section \ref{s:simulations},
and a discussion in Section \ref{s:discussion}.

\section{Preliminaries}\label{s:preliminaries}

In this section two different classes of Sobolev spaces are introduced, weighted Sobolev spaces and Kondratiev spaces. We also look at Fredholm properties for the specific operators that will appear in later sections. We will see how these properties depend on the weighted spaces used to define the domain and range of these operators. Throughout this section we use the symbol $\langle x \rangle = ( 1+ | x|^2)^{1/2}$, which appears in the definition of the norms for the weighted Sobolev spaces introduced.

\subsection{Weighted Sobolev Spaces}

Let $s$ be a nonnegative integer, $ p \in (1,\infty)$, and $ \gamma $ a real number. We denote by $W^{s,p}_{\gamma}(\R^d)$ the space of functions formed by taking the completion of  $C_0^\infty(\R^d,\C)$ under the norm
\[ \|u\|_{W^{s,p}_\gamma(\R^d)} = \sum_{|\alpha|\leq s} \| \langle x \rangle^\gamma D^\alpha u(x)\|_{L^p(\R^d)}.\]
When $p=2$ we let  $ W^{s,2}_\gamma(\R^d)= H^s_\gamma(\R^d)$. In this case these spaces are also Hilbert spaces, with inner product defined in the natural way by
 \[ \langle f, g \rangle = \sum_{|\alpha| \leq s} \int_{\R^d} f(x) \bar{g} (x) \;\langle x \rangle^{2\gamma} \;dx\]
where the overbar denotes the complex conjugate.

Notice in particular that depending on the sign of the weight $\gamma$, the functions in these spaces are either allowed to grow ($\gamma<0$ ), or forced to decay ($\gamma>0)$.
We also have natural embeddings, with $W^{s,p}_\gamma(\R^d) \subset W^{s,p}_\sigma(\R^d)$ provided $\gamma >\sigma $, and $W^{s,p}_\gamma(\R^d) \subset W^{k,p}_\gamma(\R^d)$ whenever $s>k$. For $1<p< \infty$ we can also identify the dual, $(W^{s,p}_\gamma(\R^d))^*$, with the space $W^{-s,q}_{-\gamma}(\R^d)$, where $p$ and $q$ conjugate exponents.

{\bf Kondratiev Spaces:}
With $s, p$ and $\gamma$ as in the previous section, we define Kondratiev spaces as the completion of $C_0^\infty(\R^d,\C)$ functions under the norm
 \[ \|u\|_{M^{s,p}_\gamma(\R^d)} = \sum_{|\alpha|\leq s} \| \langle x \rangle^{\gamma+|\alpha|} D^\alpha u(x)\|_{L^p(\R^d)}\]
and denote them by the symbol $M^{s,p}_{\gamma}(\R^d)$.

Again we see that these spaces are Hilbert spaces when $p=2$, with inner product given by
 \[ \langle f, g \rangle = \sum_{|\alpha| \leq s} \int_{\R^d} f(x) \bar{g} (x) \;\langle x \rangle^{2(\gamma+|\alpha|) } \;dx.\] 
We also have the following natural embeddings. One can check that $M^{s,p}_{\gamma}(\R^d) \subset M^{k,p}_\gamma(\R^d) $ whenever $s>k$, and $M^{s,p}_{\gamma}(\R^d) \subset M^{s,p}_\sigma(\R^d) $ provided $\gamma> \sigma$. 
In addition, as in the case of standard Sobolev spaces, one can identify the dual space $(M^{s,p}_{\gamma}(\R^d))^*$ with $M^{-s,q}_{-\gamma}(\R^d)$, where $p$ and $q$ are conjugate exponents.

As was the case with the weighted Sobolev spaces defined above, Kondratiev spaces encode growth or decay depending on the sign of  $\gamma$. However, in contrast to $W^{s,p}_\gamma(\R^d)$, Kondratiev spaces enforce a specific  \emph{algebraic growth} or \emph{algebraic decay} depending on the value of $\gamma$. In addition, we have the following result which summarizes decay properties for functions in $M^{s,p}_\gamma(\R^d)$ in terms of $\gamma$.

\begin{Lemma}\label{l:decay}
Let $d, m \in \N$ with $m>(d-1)/2$. Then,
 for $\gamma>-d/2$, and for all $f \in M^{m+1,2}_{\gamma}(\R^d)$  there is a constant $C>0$ such that $$|f(x)| \leq C \|\nabla f\|^\delta_{L^2_{\gamma+1}(\R^d)}\|f \|^{1-\delta}_{M^{m+1,2}_{\gamma}(\R^d)} \; |x|^{-\gamma- d/2},$$
 whenever $|x|$ is large and where $\delta = (d-1)/2m$.

\end{Lemma}

\begin{proof}
Let $(\theta, r)$ represent spherical coordinates in $\R^d$, with $r$ being the radial direction and $\theta$ representing the coordinates in the unit sphere, $\Sigma$. Then, 
\begin{align*}
 \int_\Sigma |f(\theta,R) |^2\; d\theta &
  \leq \int_\Sigma \left( \int_{\infty}^R |\partial_r f (\theta,s) | \;ds \right)^2 \;d\theta\\[2ex]
 \|f(\cdot, R)\|_{L^2(\Sigma)} &  \leq \left[ \int_\Sigma \left (  \int^{\infty}_R s^{\alpha} s^{\gamma+1} |\partial_r f (\theta,s) | s^{(d-1)/2}\;ds  \right)^2\;d\theta \right]^{1/2},
\end{align*}
where $\alpha =  -(\gamma+1)+ (1-d)/2$ and $R$ is fixed.
 Since the inner integral is squared, we also switched the bounds of integration. 
 Then, applying  Minkowski's inequality for integrals \cite[Theorem 6.19]{folland1999} followed by Cauchy Schwartz,
\begin{align*}
\|f(\cdot, R)\|_{L^2(\Sigma)} & \leq  \int^{\infty}_R s^{\alpha} \left[ \int_\Sigma   s^{2( \gamma+1)} |\partial_r f (\theta,s) |^2 s^{(d-1)}\;d\theta  \right]^{1/2}\;ds\\[2ex]
 & \leq \left[ \int^{\infty}_R s^{2\alpha} \right] ^{1/2} 
 \left[ \int_{\infty}^R   \int_\Sigma   s^{2( \gamma+1)} |\partial_r f (\theta,s) |^2 s^{(d-1)}\;d\theta   \;ds \right]^{1/2}\\[2ex] 
 & \leq C(\gamma,d) R^{\alpha +1/2} \|\nabla f\|_{L^2_{\gamma+1}(\R^d)}
 \end{align*}
 where the last line holds provided $2\alpha + 1<0$. We therefore have the following inequality 
\[ \|f(\cdot, R)\|_{L^2(\Sigma)} \leq C(\gamma, d)  R^{-\gamma-d/2}  \|\nabla f\|_{L^2_{\gamma+1}(\R^d)}, \qquad \gamma>-d/2.\]
  
Suppose now that  $f \in M^{m+1,2}_{\gamma}(\R^d)$. 
Using the chain rule we find that $|\partial_r D^k_\Sigma f(\cdot,R) | \leq |D^{k+1}f| R^k$, 
  where by $D^k_\Sigma$ we mean derivatives with respect  to unit sphere variables.
  As a result, similar calculations as the one above show that for  integer values of $k \in [1, m]$ we have that
\[ \|D_\Sigma^k f(\cdot, R)\|_{L^2(\Sigma)} \leq R^{-\gamma-d/2}  \|D^{k+1} f\|_{L^2_{\gamma+k+1}(\R^d)}, \qquad \gamma>-d/2.\]

Next we use a result from Adams and Fournier \cite[Theorem, 5.9]{adams2003sobolev},  which states that given a domain $\Sigma$ of dimension $n$, and conjugate exponents, $p,q$, satisfying $p>1$ and $mp>n$, $1\leq q \leq p$,  there exist a constant $C$ such that for all $u \in W^{m,p}(\Sigma)$
\[ \| u \|_{L^\infty(\Sigma)} \leq C \| u\|_{W^{m,p}(\Sigma)}^\delta \|u \|_{L^q(\Sigma)}^{(1-\delta)},\]
where $\delta = np/ ( np + (mp -n)q)$. Choosing $p=q=2$, $n = d-1$, $m>(d-1)/2$, we obtain $\delta = (d-1)/2m$ and
\[ \| f(\cdot, R) \|_{\infty} \leq C(m) R^{-\gamma-d/2} \|f\|^{\delta}_{M^{m+1,2}_\gamma(\R^d)} \|\nabla f\|^{1-\delta}_{L^2_{\gamma+1}(\R^d)},\]
provided $f \in M^{m+1,2}_{\gamma}(\R^d)$.

\end{proof}

\begin{Remark}
Although in the definitions presented above, the spaces $M^{s,p}_\gamma(\R^d)$ and $W^{s,p}_\gamma(\R^d)$ consist of complex-valued functions, in what follows we will assume that all functions are real-valued.
\end{Remark}

\subsection{Fredholm operators}
In this section and throughout the paper, we use the notation $M^{s,p}_{r,\gamma}(\R^d)$ and $ H^s_{r,\gamma}(\R^d)$ to denote the subspaces of radially symmetric functions in $M^{s,p}_{\gamma}(\R^d)$ and $H^s_\gamma(\R^d)$, respectively.

The next Lemma shows that for $\lambda > 0$, the operator $\partial_r + \frac{1}{r} +\lambda$ is invertible in appropriate spaces.

\begin{Lemma}\label{l:Llambda}
Let $\gamma \in \R$,  $\lambda > 0$, and $k \in \N$.
Then, the operator $(\partial_r + \frac{1}{r} + \lambda): H^k_{r,\gamma}(\R^2) \longrightarrow H^{k-1}_{r,\gamma}(\R^2)$   has a bounded inverse.
\end{Lemma}
\begin{proof}
A short calculation shows that the kernel of this operator is spanned by the function  $\rme^{-\lambda r}/r$, which is singular near the origin and is therefore not in $L^2_{r, \gamma}(\R^2)$.  At the same time, the adjoint of this operator is given by $-\partial_r + \lambda : L^2_{r,-\gamma}(\R^2) \longrightarrow H^{-1}_{r,-\gamma}(\R^2)$, and we find that the cokernel is spanned by the function $\rme^{\lambda r}$, which again is not in the space $L^2_{r,-\gamma}(\R^2) $ no matter what the value of $\gamma $ is. As a result, the kernel and co-kernel of this operator are trivial.

To prove the result, we are left with showing that 
that the inverse operator 
\[\begin{array}{c c c}
L^2_{r, \gamma}(\R^2) & \longrightarrow & H^1_{r, \gamma}(\R^2)\\
f(r) & \mapsto & u(r) =\frac{1}{r} \int_0^r  \rme^{\lambda( s-r)} f(s) s \;ds
\end{array}
\]
is bounded. To show that $\|u \|_{L^2_{r,\gamma}(\R^2)} \leq \| f\|_{L^2_{r,\gamma}(\R^2)}$, we use the inequality
\[ \|u \|_{L^2_{r,\gamma}(\R^2)} \leq \|u \|_{L^2_{r,\gamma}(B_1) } + \|u \|_{L^2_{r,\gamma}(\R^2\backslash B_1)}, \]
where $B_1$ is the unit ball in $\R^2$. Then
\begin{align*}  
\|u \|_{L^2_{r,\gamma}(\R^2\backslash B_1)}  
= & \left [  \int_1^\infty \left| \int_0^r f(s) \rme^{\lambda(s-r)} \frac{s}{r} \;ds   \right|^2  \langle r \rangle^{2 \gamma } r\;dr \right]^{1/2}  \\
\leq &  2^{|\gamma +1/2|}  \left [  \int_1^\infty \left| \int_0^r f(s)  \langle s \rangle^{\gamma  + 1/2}\quad   \rme^{-\lambda(r-s)}  \langle r- s\rangle^{| \gamma + \frac{1}{2} |}  \;ds \right|^2  \;dr \right]^{1/2},\\
\leq &   2^{|\gamma +1/2|}   \left [  \int_0^\infty \left|  \int_0^r |f(r-z)| \langle r-z \rangle^{\gamma+1/2}  \rme^{-\lambda z}  \langle z \rangle^{|\gamma+1/2|}
\;dz \right|^2\;dr \right]^{1/2}\\
\leq &  2^{|\gamma +1/2| \int_0^\infty e^{-\lambda z} \langle z \rangle^{|\gamma+1/2|}
\left(   \int_z^\infty |f(r-z)|^2 \langle r-z \rangle^{2\gamma} z \;dz \right)^{1/2}  \;dz  } \\
\leq & C_1(\gamma,\lambda)\| f\|_{L^2_{r,\gamma}(\R^2)} 
\end{align*}
where the second line follows from 
the fact that $s/r<1$ and the relation  $ \langle s \rangle^{-\sigma} \langle r \rangle^\sigma \leq 2^{|\sigma|} \langle r-s \rangle^{|\sigma|}$.
 The  inequality on the third line comes from  using the change of coordinates $z = r-s$ and extending the outer limits of integration to zero,
while the fourth line follows 
from an application of  Minkowski's inequality for integrals, \cite[Theorem 6.19]{folland1999}.
In the final result  we let $C_1(\gamma, \lambda) = 2^{|\gamma +1/2|} \int_0^\infty e^{-\lambda z} \langle z \rangle^{|\gamma+1/2|}\;dz$.

To prove the relation $\|u \|_{L^2_{r,\gamma}(B_1) }  \leq C \| f\|_{L^2_{r,\gamma}(\R^2)}$,
we bound

\begin{align*}
 |u(r)| & \leq \frac{1}{r} \int_0^r \left | f(s) \rme^{-\lambda(r- s)} s \right| \;ds \\
  & \leq \frac{1}{r} \int_0^r \left | f(s)  \right| s \;ds \\
 & \leq \frac{1 }{r} \left( \int_0^r  | f(s)|^2  s \;ds\right)^{1/2} \left( \int_0^r s \;ds \right)^{1/2}\\
 & \leq \frac{1 }{r}  \left( \int_0^\infty  | f(s)|^2 \langle s \rangle^{2 \gamma}  s \;ds\right)^{1/2}   \left( \frac{r}{\sqrt{2}}  \right)\\
 & \leq \frac{1}{\sqrt{2}} \| f\|_{L^2_{r,\gamma}(\R^2)}
\end{align*}
where the third line follows from H\"older's inequality. We then obtain that
$ \| u(r)\|_{L^2_{r,\gamma}(B_1) }\leq  C_2(\gamma) \;\| f\|_{L^2_{r,\gamma}(\R^2)}$
with $C_2(\gamma) = \frac{1}{\sqrt{2}} \left( \int_0^1 \langle r\rangle^{2\gamma} r \;dr \right)^{1/2}$,
and consequently
\[ \| u(r)\|_{L^2_{r,\gamma}(\R^2) }\leq (C_1(\gamma,\lambda) + C_2(\gamma) ) \| f\|_{L^2_{r,\gamma}(\R^2)}.\]

Next, to show that the derivative $\partial_r u \in L^2_{r,\gamma}(\R^2)$, we use the equation to 
write $\partial_r u = f - \lambda u - \frac{u}{r}$, and thus obtain
\[ \| \partial_r u \|_{L^2_{r,\gamma}(\R^2 \setminus B_1)} \leq (1 +(1+\lambda) C_1(\gamma, \lambda))\; \| f\|_{L^2_{r,\gamma}(\R^2)}.\]
To bound $\| \partial_r u \|_{L^2_{r,\gamma}(B_1)}$, notice first that
\[  \frac{|u(r)|}{r}  \leq \frac{1}{r^2} \int_0^r |f(s)| s\;ds \leq  \frac{1}{r^2} \int_0^r |f(s)- f(y)|s \;ds + \frac{1}{2} |f(y)|,\]
where we pick $y \in [0,r]$. 
Letting $\tilde{B} (y,2r)  = B(y,2r) \cap B(0,r) \subset \R^2$ , where $B(y,2r)$ is the ball centered at $y$ of radius $2r$,
the inequality becomes
\[  \frac{|u(r)|}{r} \leq \frac{|\tilde{B}(y,2r)|}{2 \pi r^2} \left( \frac{1}{|\tilde{B}(y,2r)|} \int_{\tilde{B}(y,2r)} |f(s) -f(y)| s \;ds\;d \theta \right) + \frac{1}{2} |f(y)|.\]
Here, $|\tilde{B}(y,2r)|$ denotes the measure of the set $\tilde{B}(y,2r)$.
Since $L^2(B(0,M)) \subset L^1(B(0,M))$ for any ball $B(0,M) \subset \R^2$, with finite radius $M$, we have that $f$ is in $L^1_{loc}(B(0,M))$.
By the Lebesgue Differentiation Theorem \cite[Theorem 3.21]{folland1999}, 
the expression in parenthesis approaches zero as $r\to 0$, while the fraction in front remains bounded
since $|\tilde{B}(y,2r)|>2\pi r^2$. 
Therefore, close to the origin, the function $|u(r)|/r$ is bounded by $f(r)$ and, using again the equation $\partial_r u = f - \lambda u - u/r$, we find that
 \[ \| \partial_r u \|_{L^2_{r,\gamma}(B_1)} \leq (2 +(1+\lambda) C_2(\gamma))\; \| f\|_{L^2_{r,\gamma}(\R^2)}.\]
It then follows that
\[ \| \partial_r u\|_{L^2_{r,\gamma}(\R^2)} \leq  (3 + (1+\lambda)(C_1(\gamma, \lambda) + C_2(\gamma))) \| f\|_{L^2_{r,\gamma}(\R^2)}.\]

The above calculations then show that the map $\mathcal{L}_\lambda= \partial_r + \frac{1}{r} +\lambda : H^1_{r,\gamma}(\R^2) \longrightarrow L^2_{r,\gamma}(\R^2)$ is invertible. Moreover, we also obtain that the operator norm of its inverse satisfies,
\[ \| \mathcal{L}^{-1}_\lambda\| \leq 3 + (2+\lambda)( C_1(\gamma,\lambda) + C_2(\gamma) ).\]

To extend the result to the more general operator $\partial_r + \frac{1}{r} +\lambda : H^k_{r,\gamma}(\R^2) \longrightarrow H^{k-1}_{r,\gamma}(\R^2)$ one can proceed by induction:
 Assuming that $f$ and $u$ are in $ H^{k-1}_{r,\gamma}(\R^2)$  one shows that $\partial_r^k u $ is in $L^2_{r,\gamma}(\R^2)$ using the relation $\partial^k_r u = \partial_r^{k-1} (f - u - \frac{u}{r})$. The fact that $\partial_r^{k-1} \left(\dfrac{u}{r}\right)$ is in the correct space follows by a similar argument as the one done above to prove $u/r \in L^2_{r,\gamma}(\R^2)$.

\end{proof}

To simplify notation, we define $ \mathcal{L}(\lambda) = \partial_r + \frac{1}{r} + \lambda$ and prove in the next Lemma that its inverse, defined over appropriate weighted spaces, is continuously differentiable with respect to the parameter $\lambda$.

\begin{Lemma}\label{l:contlambda}
Let $\lambda> 0$ and $k \in \N \cup\{0\}$ and consider the operator defined by $\mathcal{L}(\lambda) u = \partial_r u + \frac{1}{r} u + \lambda u$. Then,  its inverse, 
\[
\begin{array}{c c c c}
\mathcal{L}^{-1}(\lambda) : & H^k_{r, \gamma}(\R^2) & \longrightarrow & H^{k+1}_{r, \gamma}(\R^2)\\
\end{array}
\]
is $C^1$ with respect to the parameter $\lambda$. 
\end{Lemma}

\begin{proof}
Lemma \ref{l:Llambda}  shows that $\mathcal{L}^{-1}(\lambda)$, with the specified domain and range, is a bounded operator for all $\lambda \in (0,\infty)$. To prove the continuity of this operator with respect to $\lambda$ we must show that given $f \in H^k_{r,\gamma}(\R^2)$,
\[ \sup_{\|f\|_{H^k_{r,\gamma}} = 1} \| (\mathcal{L}^{-1}(\lambda + h) - \mathcal{L}^{-1}(\lambda) ) f \|_{H^{k+1}_{r,\gamma}} \leq C h. \]

Using the notation $\phi(\lambda) = \mathcal{L}^{-1}(\lambda) f $, we notice that
\begin{align*}
 (\mathcal{L}^{-1}(\lambda + h) - \mathcal{L}^{-1}(\lambda) ) f  
 = & \phi(\lambda + h) -\phi(\lambda)\\
=& - \mathcal{L}^{-1}(\lambda) \left[  (\mathcal{L}(\lambda + h) - \mathcal{L}(\lambda) ) \right] \phi(\lambda +h)\\
= & - h \mathcal{L}^{-1}(\lambda)  \mathcal{L}^{-1}(\lambda+h )f,
\end{align*}
from which the desired result follows.
This last expression also shows that for $\lambda > 0$, the derivative of $\mathcal{L}^{-1}(\lambda)$ with respect to this parameter is given by
\[
\begin{array}{c c c c}
\partial_\lambda \mathcal{L}^{-1}(\lambda): & H^k_{r, \gamma}(\R^2) & \longrightarrow &H^{k+1}_{r,\gamma}(\R^2) \\
& f& \mapsto  & - \mathcal{L}^{-1}(\lambda)\mathcal{L}^{-1}(\lambda) f
\end{array}
\]
Since the derivative $\partial_\lambda \mathcal{L}^{-1}(\lambda)$ is the composition of two continuous operators, it follows that it is itself continuous with respect to $\lambda$.
\end{proof}

Finally, the next proposition establishes Fredholm properties for the radial operators $\Delta_n: M^{2,2}_{r, \gamma-2}(\R^2) \rightarrow L^2_{r,\gamma}(\R^2)$ ,
\[ \Delta_n = \partial_{rr}  + \frac{1}{r} \partial_r  - \frac{n^2}{r^2}, \quad n \in \N \cup \{ 0\}.\]
The results follows from 
  \cite{mcowen1979}, where it is shown that the Laplace operator
$\Delta : M^{2,p}_{\gamma-2}(\R^2) \rightarrow L^p_\gamma(\R^2)$ is Fredholm,
and the fact that when $p=2$, one can decompose the space $M^{2,2}_{\gamma-2}(\R^2)$ into
 a direct sum $\oplus m^2_{n,\gamma-2}$ where
  \[ m^2_{n,\gamma-2} =\{ u \in M^{2,2}_{\gamma-2}(\R^2): \; u(r, \theta) = v_n(r) \rme^{\rmi n \theta}, 
 \quad v_n \in M^{2,2}_{r,\gamma-2}(\R^2)\}, \qquad n \in \Z. \]
 Notice that because the functions in $M^{2,2}_{\gamma-2}(\R^2)$ are real valued we have that $ v_{-n} = \bar{v}_n$. 
For a detailed proof of this next result, see \cite[Lemma 3.1]{jaramillo2018}.

\begin{Proposition}\label{p:Delta_nFredholm}
Let $\gamma \in \R \backslash \Z$, and $n \in \Z$. Then, the operator $\Delta_n: M^{2,2}_{r, \gamma-2}(\R^2) \rightarrow L^2_{r,\gamma}(\R^2)$ given by 
\[ \Delta_n \phi = \partial_{rr} \phi + \frac{1}{r} \partial_r \phi - \frac{n^2}{r^2} \phi\]
is a Fredholm operator and,
\begin{enumerate}
\item for $1-|n| < \gamma< |n|+1$, the map is invertible;
\item for $\gamma>|n|+1 $, the map is injective with cokernel spanned by $r^{|n|}$;
\item for $\gamma< 1-|n|$, the map is surjective with kernel spanned by $r^{|n|}$.
\end{enumerate}
On the other hand, the operator is not Fredholm for integer values of $\gamma$.
\end{Proposition}


\section{Intermediate Approximations to the Viscous Eikonal Equation}\label{s:inter}

As mentioned in the introduction, our interest in the viscous eikonal equation
 \[\tilde{\phi}_t = \Delta \tilde{ \phi} - b |\nabla \tilde{\phi } |^2 - \eps g(x), \quad x \in \R^2,\]
 stems from  its role as a model equation for the phase dynamics of 
target patterns and spiral waves in oscillatory systems. 
We are therefore  interested in solutions of the form $\tilde{\phi}(x,t) = \phi (x) - \Omega t$,
which then satisfy the steady state equation,
\begin{equation}\label{e:eikonal}
\Delta \phi - b|\nabla \phi |^2 - \eps g(x) + \Omega =0 \qquad x\in \R^2.
\end{equation} 
Because the gradient, $\nabla \phi$, then approximates the pattern's wavenumber, 
target patterns then correspond to those $\phi$ which in addition fulfill the
boundary conditions,
$\nabla \phi \to k$ as $|x| \to \infty$.
Consequently, we look for solutions to equation \eqref{e:eikonal}
that bifurcate from zero when $\eps>0$,  and whose gradients are bounded at infinity.

Notice that the condition on the gradient, $\nabla \phi$, provides enough information to derive an equation that is valid at intermediate scales. Indeed, assuming a regular perturbation for both $\phi$ and $\Omega,$ one obtains at order $\rmO(\eps)$ the equation,
\[ \partial_{rr} \phi_1 + \frac{1}{r} \partial_r \phi_1 -  g =  - \Omega_1.\]
A short calculation then shows that  in order to obtain solutions with bounded derivatives, the parameter $\Omega_1$ must be zero. Continuing this  perturbation analysis one checks that this condition must be satisfied at all orders of $\eps$ . In other words, the frequency, $\Omega$, must be small beyond all orders of this parameter.
Consequently, at intermediate scales the system is well approximated by the intermediate equation
\begin{equation}\label{e:intermediate}
\Delta_0 \phi - b(\partial_r \phi)^2 - \eps g_c - \eps g_f = 0.
\end{equation}

Notice that we have explicitly written the inhomogeneity as the sum of two functions satisfying Hypothesis \ref{h:g}. This choice of notation will be used next in Subsection \ref{ss:firstapprox}, where
we construct a first order approximation for the above equation. We then  use this information
to prove existence of solutions to equation ~\eqref{e:intermediate} in Subsection \ref{e:existenceInter}.

{\bf Notation:} Throughout this section, and in the rest of the paper, we use $\gamma_\epsilon$ to denote
the Euler Mascheroni constant, and  the symbols $\chi, \chi_M\in C^\infty(\R^2)$ to denote smooth radial cut-off functions satisfying
\[
\chi(x) = \left \{ 
\begin{array}{c c c}
0 & \mbox{if} & |x|<1\\
1 & \mbox{if} & |x|>2
\end{array},
\right.
\qquad
 \chi_M(x) = \left \{ 
\begin{array}{c c c}
0 & \mbox{if} & |x|<1\\
1 & \mbox{if} & 2 < |x|< M\\
0 & \mbox{if} & 2M <|x|\\
\end{array},
\right.
\]
 where $M$ is a positive constant.  Notice in particular, that $\chi_M$ has compact support.

\subsection{First Order Approximation}\label{ss:firstapprox}

We construct a first order approximation, $\phi$, which is the sum of two functions, $\phi_0$ and  $\phi_1$. 
We take 
\begin{equation}\label{e:approx0}
 \phi_0 = - \frac{1}{b} \chi_M \log( 1+ a \log r + \eps K(r)),
 \end{equation}
a choice that is motivated by the Hopf-Cole transform, $\phi = -\frac{1}{b} \log \Psi$, which turns the eikonal equation \eqref{e:intermediate} into the steady state Schr\"odinger equation with potential $\eps g$. 
The value of the constant $M$ appearing in the definition of the cut-off function $\chi_M$ is taken so that the
expression $\log( 1+ a \log r + \eps K(r)),$ always remains bounded.
The constant $a$ is a parameter that is determined when constructing the second part to the approximation,
 while the function $K(r)$ satisfies 
\begin{equation}\label{e:K}
  \Delta_0 K + b g_f =0.
  \end{equation}

The fact that we can solve this last equation follows from our assumptions on the inhomogeneity. Recall that $g_f \in H^k_\sigma(\R^2)$ for $0<\sigma<1$. Proposition \ref{p:Delta_nFredholm} then shows that the radial Laplacian, $\Delta_0$, is a surjective operator with a one dimensional kernel spanned by $\{1\}$.
We can therefore use Lyapunov-Schmidt reduction to solve this equation and find a family of solutions 
\[ K(r) = K_p(r) + c, \qquad c \in \R.\]
Since $c$ is arbitrary, without loss of generality we pick $c=0$.
We also find  that the solution, $K$, belongs to the space $ M^{2,2}_{r, \sigma-2}(\R^2)$. In fact, one can check that $K$ has more regularity and is in the space
\begin{equation}\label{e:spaceR}
 R_\sigma^k = \{ u \in M^{2,2}_{r, \sigma-2}(\R^2) : D^2 u \in H^k_\sigma(\R^2)\}.
 \end{equation}
 
\begin{Remark}\label{r:phi0}
Notice that the function $\phi_0$ is as regular as the function $K$, and that as a result the derivative
$\partial_r \phi_0$ is in the space $H^{k+1}_\sigma(\R^2)$.
In addition, this function is bounded and has compact support.
\end{Remark}
\begin{Remark}\label{r:K0}
Because {\color{blue} $g_f$} is in $H^k_\sigma(\R^2)$ with $\sigma \in (0,1)$ and $k\geq2$,
we then have the following decay properties for the solution  to equation 
\eqref{e:K}:
\begin{itemize}
\item if {\color{blue} $g_f $} decays like $1/r^m$ in the far field, with $1<m<2,$ then $K \sim \rmO(r^{2-m})$ at infinity, while
\item if {\color{blue} $g_f$} decays like $1/r^2$ in the far field, then $K \sim \rmO( (\log r)^2)$ at infinity.
\end{itemize}
\end{Remark}

Next, we define the second function, $\phi_1$,  
as the solution to the equation
\[ \Delta_0 \phi_1 - \frac{a}{b} \Delta_0 ( \chi \log r) - \eps g_c =0.\]
Here, the constant $a$  
is the same as the one 
 appearing in the definition of $\phi_0$,
 and the function $g_c $ is in the space $ H^k_\gamma(\R^2)$ with $1<\gamma$, by assumption.

To justify the existence of  $\phi_1$, we use again
Proposition \ref{p:Delta_nFredholm} which shows that for values of $\gamma>1$,
 the operator $\Delta_0:M^{2,2}_{r, \gamma-2}(\R^2) \longrightarrow L^2_{r,\gamma}(\R^2)$ is Fredholm with index -1, and cokernel spanned by $\{1 \}$.
Because the projection of  $\Delta_0 ( \chi \log r)$ onto the cokernel is non-trivial, i.e.
\[ \int_{\R^2} \Delta_0 ( \chi \log r) \;dx = 2\pi,\]
the Bordering Lemma stated at the end of this subsection 
then shows that the operator 
\[\begin{array}{c c c}
M^{2,2}_{r,\gamma-2}(\R^2) \times \R & \longrightarrow & L^2_{r, \gamma}(\R^2)\\[2ex]
(\phi, a) & \longmapsto & \Delta_0 \phi_1 - \dfrac{a}{b} \Delta_0 ( \chi \log r) 
\end{array}\]
is invertible. 
Therefore, the equation for $\phi_1$ is indeed solvable. 
In addition, projecting onto the constant functions we also find that
\begin{equation}\label{e:constantA} a=  -\eps b \int_0^\infty g_c(r) \;r \;dr. \end{equation}
Finally, since $g_c \in H^k_{\gamma}(\R^2)$, it follows that our solution $\phi_1$ is in the space $R_\gamma^k$, defined as in \eqref{e:spaceR}.

\begin{Lemma}\label{e:bordering}[Bordering Lemma]
Let $X$ and $Y$ be Banach spaces, and consider the operator
\[ S = \begin{bmatrix}
A & B\\
C& D
\end{bmatrix} : X \times \R^p \longrightarrow Y \times \R^q,\]
with bounded linear operators $A: X \longrightarrow Y$, $B:\R^p \longrightarrow Y$, $C: X  \longrightarrow \R^q$,
$D: \R^p \longrightarrow \R^q$. If $A$ is Fredholm of index $i$, then $S$ is Fredholm of index $i +p-q$.
\end{Lemma}

\begin{proof}
One can write $S$ as the sum of a block diagonal operator with the indicated index, $i + p-q$,
and a compact operator  consisting of the off-diagonal elements.
Since compact perturbations do not alter the index of a Fredholm operator, the result then follows.
\end{proof}

\subsection{Existence of Solutions to Intermediate Approximation}\label{e:existenceInter}
Using the first order approximation, $\phi_0+\phi_1$, defined in the previous subsection we now prove the existence of solutions  to equation \eqref{e:intermediate} 
using the implicit function theorem. 

Inserting the ansatz
\[ \phi = \phi_0 + \phi_1 + \phi_2\]
into equation \eqref{e:intermediate},
 one obtains  the following  expression for $\phi_2$,
\begin{equation}\label{e:inter2}
 \Delta_0 \phi_2 -b \Big ( 2 \partial_r \phi_0( \partial_r \phi_1 + \partial_r \phi_2) + ( \partial_r \phi_1 + \partial_r \phi_2)^2 \Big) + G_1 =0,
\end{equation}
where the term $G_1 $ is given by
\[G_1  = \Delta_0 \phi_0 - b( \partial_r \phi_0)^2 - \eps g_f +\frac{a}{b} \Delta_0 ( \chi \log r) \]

To continue the analysis, we let $\psi = \partial_r \phi_2$ in equation \eqref{e:inter2}, 
and add and subtract the term $\lambda \psi$.
We assume that the parameter $\lambda$ is sufficiently small, positive, and fixed. The result is,
\[\partial_r \psi + \frac{1}{r} \psi + \lambda \psi -b\Big ( 2 \partial_r \phi_0( \partial_r \phi_1 + \psi) + ( \partial_r \phi_1 + \psi)^2 \Big) + G_1 - \lambda \psi =0.\]
Letting $\mathcal{L}_\lambda = \partial_r + \frac{1}{r} + \lambda $, we may precondition 
 this last equation with $\mathcal{L}^{-1}_\lambda$ and write
\begin{equation}\label{e:interIFT}
\psi + \mathcal{L}_\lambda^{-1}\left[ -b \Big ( 2 \partial_r \phi_0( \partial_r \phi_1 + \psi) + ( \partial_r \phi_1 + \psi)^2 \Big) + G_1 - \lambda \psi  \right] =0.
\end{equation}
Because the above expression is equivalent to the intermediate equation \eqref{e:intermediate}, if we find a solution $\psi$ to \eqref{e:interIFT}, we immediately obtain a corresponding solution to \eqref{e:intermediate} of the form
\[ \phi(r;\eps) = \phi_0(r; \eps) + \phi_1(r;\eps) + \phi_2(r;\eps) +\eps c, \]
where $\partial_r \phi_2 = \psi$ and $c$ is a constant of integration.

 To use the implicit function theorem, we  view the left hand side of \eqref{e:interIFT}  as an operator $F: H^k_{\delta} (\R^2) \times \R_+ \longrightarrow H^k_{\delta} (\R^2) $ for some appropriate $\delta \in \R$, and show that it is well defined,  smooth with respect to $\eps$, and that its Fr\'echet derivative $D_\psi F(0;0): H^k_{\delta} (\R^2) \longrightarrow H^k_{\delta} (\R^2)$ is invertible.
  The result is the following theorem.

\begin{Theorem}\label{t:inter}
Let $k \geq 2$, $\sigma \in (0,1)$ and consider functions $ g \in H^k_{r,\sigma}(\R^2)$ satisfying Hypothesis \ref{h:g}. Then, the intermediate equation
\[ \Delta_0 \phi - b(\partial_r \phi)^2 - \eps g =0,\]
has a family of solutions 
\[\phi(r; \eps) = -\frac{1}{b}\chi_M \log\Big( 1 + a \log r + \eps K \Big) + \phi_1(r;\eps) + \phi_2(r;\eps)+ \eps c, \qquad c \in \R,\]
 that bifurcates from  zero at $\eps =0$ and  is $C^1$ in $\eps \in [0,\infty)$.
 Moreover, letting $\gamma \in (1, \infty)$, and  $\sigma \in (0,1)$ be defined as in Hypothesis \ref{h:g}, 
 the family of solutions satisfies:
\begin{itemize}
\setlength \itemsep{1ex}
\item $\phi_1 \in \{ u \in M^{2,2}_{r,\gamma-2}(\R^2) : D^2u \in H^k_\gamma(\R^2)\}$, 
\item $\partial_r \phi_2 \in H^k_\delta(\R^2)$, where $\delta = \min(\gamma-1, \sigma)$.
\item $\Delta_0 K = g_f$, with $K \in R^k_\gamma$, and
\item $ a = -\eps b \int_0^\infty g_c(r) \;r \;dr$.
\end{itemize}
\end{Theorem}
\begin{proof}
As already mentioned, the result follows from finding solutions to equation \eqref{e:interIFT} using the implicit function theorem. We therefore consider the left hand side of this equation as an operator $F: H^k_{\delta} (\R^2) \times \R_+ \longrightarrow H^k_{\delta} (\R^2)$, with $\delta = \min(\gamma-1, \sigma) >0$.

Since the operator's dependence on $\eps$ comes from the three functions $\partial_r \phi_0, \partial_r \phi_1,$ and $G_1$, and since these functions are all smooth with respect to $\eps$ on the interval $[0,\infty)$, then the same result holds for the operator $F$.

To show that the Fr\'echet derivative $D_\psi F(0;0) = Id - \lambda \mathcal{L}_\lambda^{-1}$ is invertible, we recall the results from Section \ref{s:preliminaries}. In particular, Lemma \ref{l:Llambda} shows that $\mathcal{L}_\lambda^{-1}: H^k_{\delta}(\R^2) \longrightarrow H^{k+1}_{\delta}(\R^2)$ is bounded. Since the embedding $H^{k+1}_{\delta}(\R^2) \subset H^{k}_{\delta}(\R^2)$ is continuous, it then follows that  $D_\psi F(0;0) :  H^k_{\delta}(\R^2) \longrightarrow H^k_{\delta}(\R^2)$  is a small perturbation of the identity operator, and is thus invertible for a sufficiently small $\lambda$.

To complete the proof we need to show that the operator $F$ is well defined.
Taking into account the results of Lemma \ref{l:Llambda}, this is equivalent to showing that the expression
\[N(\psi,\eps) = -b \Big ( 2 \partial_r \phi_0( \partial_r \phi_1 + \psi) + ( \partial_r \phi_1 + \psi)^2 \Big) + G_1 - \lambda \psi,\]
defines a bounded operator 
$N: H^k_{\delta}(\R^2) \times \R_+ \longrightarrow H^{k-1}_{\delta}(\R^2)$.

We start with the term  $\partial_r \phi_0( \partial_r \phi_1 + \psi)$. 
From the definition of $\phi_1$ we know that this is a function in $R^k_\gamma$ with $\gamma>1$. In particular,
\[ \partial_r \phi_1 \in \{ u \in M^{1,2}_{r,\gamma-1}(\R^2) : Du \in H^k_{\gamma}(\R^2) \} \subset H^{k+1}_{\gamma-1}(\R^2).\]
Because $\psi \in H^k_\delta(\R^2)$ and $\delta = \min(\gamma-1, \sigma) >0$ , it then follows that the sum $( \partial_r \phi_1 + \psi)$ is also in this space.
Since $\partial_r \phi_0$ has compact support and $k+1$ bounded derivatives (see Remark \ref{r:phi0}), then the product
 $\partial_r \phi_0( \partial_r \phi_1 + \psi)$ is also well defined in $H^k_{\delta}(\R^2)$.

Next, since $( \partial_r \phi_1 + \psi)$ is in $H^k_\delta(R^2)$, with $\delta>0$ and $k \geq 2$, Lemma \ref{l:product} below shows that $( \partial_r \phi_1 + \psi)^2$ is in $H^{k-1}_\delta(\R^2)$. Finally, 
Lemma \ref{l:G1} at the end of this section shows that $G_1 \in H^k_\sigma(\R^2)$, and because $\delta = \min(\gamma-1, \sigma)>0$, this term is also well defined. 

Since the operator $F$ satisfies the assumptions of the implicit function theorem we obtain a family of solutions $\psi(r; \eps)$ that bifurcates from zero and is smooth with respect to $\eps$.
Because $\psi = \partial_r \phi_2$,
we arrive at the family
\[ \phi(r;\eps) = \phi_0(r; \eps) + \phi_1(r;\eps) + \phi_2(r;\eps) +\eps c, \quad c\in \R. \]
This finishes the proof of the Theorem.

\end{proof}

\begin{Lemma}\label{l:product}
Let $\psi \in H^k_\gamma(\R^2)$ with $\gamma >0$ and $k \geq 2$. Then,  $\psi^2 \in H^{k-1}_\gamma(\R^2)$.
\end{Lemma}

\begin{proof}
To simplify the analysis we let
 $D^j$ denote any $j$-th order derivative, and
we only prove that $D^{k-1}(\psi^2) $ is in $L^2_{r,\gamma}(\R^2)$, since a similar analysis shows that lower
derivatives are in this same space. Because $k \geq 2$ and $\gamma>0$, it follows by Sobolev embeddings that
$D^{j}\psi \in H^2_{r,\gamma}(\R^2) \subset H^2(\R^2) \subset C_B(\R^2)$ for  $0\leq j \leq k-2$. Then, writing
\[ D^{k-1}(\psi^2) = \sum_{j=0}^{k-1} {k-1 \choose j} D^{k-1-j}\psi D^j \psi \]
 we see that this derivate can be written as a product of a bounded function and a function that is in $L^2_{r,\gamma}(\R^2)$. Hence $\psi^2 \in H^{k-1}_{r,\gamma}(\R^2)$.
\end{proof}

The next Lemma shows that $G_1$ is in $H^k_\sigma(\R^2)$ with $\sigma \in (0,1)$.
\begin{Lemma}\label{l:G1}
Let $k\geq 2$, $\sigma \in (0,1)$ and take $g_f \in H^k_\sigma(\R^2)$. Consider the function $\phi_0$ constructed from $g_f$ and described above in \eqref{e:approx0}. Then the expression
\[G_1  = \Delta_0 \phi_0 - b( \partial_r \phi_0)^2 - \eps g_f +\frac{a}{b} \Delta_0 ( \chi \log r) \]
is also in $H^k_\sigma(\R^2)$.
\end{Lemma}
\begin{proof}
Using the notation $\phi_0 = \chi_M \tilde{\phi}_0$, we first expand $G_1$ 
\begin{align*}
G_1 = &\Delta_0 \phi_0 - b( \partial_r \phi_0)^2 - \eps g_f +\frac{a}{b} \Delta_0 ( \chi \log r) \\[1ex]
G_1  = &  \left( \tilde{\phi}_0 \Delta_0 \chi_M  + 2 \chi_M' \tilde{\phi}_0 
-b ( \chi_M' \tilde{\phi}_0)^2 - 2b \chi_M'  \chi_M \tilde{\phi}_0 \partial_r \tilde{\phi}_0
+b ( \partial_r \tilde{\phi}_0)^2 ( \chi_M- \chi_M^2) \right) \\
& - \frac{1}{b} \chi_M\left[ \frac{ \Delta_0 ( a \log r + \eps K)}{1 + a \log r + \eps K}\right] - \eps g_f + \frac{a}{b}\Delta_0 ( \chi \log r)   \\[1ex]
G_1  = &  \left( \tilde{\phi}_0 \Delta_0 \chi_M  + 2 \chi_M' \tilde{\phi}_0 
-b ( \chi_M' \tilde{\phi}_0)^2 - 2b \chi_M'  \chi_M \tilde{\phi}_0 \partial_r \tilde{\phi}_0
+b ( \partial_r \tilde{\phi}_0)^2 ( \chi_M- \chi_M^2) \right) \\
& - \frac{1}{b} \chi_M \left[ \frac{ \Delta_0  a \log r }{1 + a \log r + \eps K} \right]  - \frac{1}{b} \chi_M\left[ \frac{ \eps \Delta_0  K }{1 + a \log r + \eps K}  \right] - \eps g_f +  \frac{a}{b} \Delta_0 (\chi \log r). \\[2ex]
\end{align*}

Because the $\log r$ is a fundamental solution of the Laplacian and since $\chi_M$ is zero near the origin, then the term
\[\chi_M \left[ \frac{ \Delta_0  a \log r }{1 + a \log r + \eps K} \right] =0.\]
Similarly, because the function $K$ is a solution to $\Delta_0 K +  b g_f=0$, then we may write
\begin{align*}
-\frac{1}{b} \chi_M\left[ \frac{ \eps \Delta_0  K }{1 + a \log r + \eps K}   \right] - \eps g_f  = & - \eps g_f (1- \chi_M) - \eps g_f \chi_M\left[ \frac{ ( a \log r + \eps K)  }{1 + a \log r + \eps K}  \right].
\end{align*}
Therefore,
\begin{align*}
G_1  = &  \left( \tilde{\phi}_0 \Delta_0 \chi_M  + 2 \chi_M' \tilde{\phi}_0 
-b ( \chi_M' \tilde{\phi}_0)^2 - 2b \chi_M'  \chi_M \tilde{\phi}_0 \partial_r \tilde{\phi}_0
+b ( \partial_r \tilde{\phi}_0)^2 ( \chi_M- \chi_M^2) \right) \\
& - \eps g_f (1- \chi_M) - \eps g_f \chi_M\left[ \frac{ ( a \log r + \eps K)  }{1 + a \log r + \eps K}  \right]
+  \frac{a}{b} \Delta_0 (\chi \log r) 
\end{align*}

From the definition of $\chi_M$ it is clear that all terms involving a derivative of this function are localized
and have compact support. 
Because the value of $M$ in the definition of $\chi_M$ was chosen
 to vanish whenever the expression $1 + a \log r + \eps K $ is $\leq 0$,
 we see that the term in brackets is also bounded
and with compact support. In addition, this term is as regular as the function $K \in R^k_\sigma$, and
is therefore in $H^{k+2}(\R^2)$.
On the other hand, the function $(1- \chi_M) \eps g_f$ behaves like $g_f$ at infinity and as a result it is in 
the same space as the inhomogeneity. Finally,  because $ \Delta_0 \log r = 0 $ on $\R^2 \setminus \{0\}$, the function $\Delta_0 (\chi \log r) $ is localized and smooth.
Taking all this into account, we may conclude that $G_1$ is in the space $H^k_\sigma(\R^2)$.

\end{proof}

\section{Far Field Approximation to the Viscous Eikonal Equation}\label{s:far}
In this section we consider again the full equation
\begin{equation}\label{e:far}
\Delta_0 \phi - b(\partial_r \phi) ^2 - \eps g(r) + \Omega =0 \qquad r= |x| \in [0,\infty),
\end{equation}
but assume that the value of $\Omega$ is fixed and different from zero. As in Section \ref{s:inter},
we first find 
 an appropriate expression for the far field behavior of the solution and a
first order approximation for this new equation.
We then use this result to prove existence 
of solutions using the implicit function theorem. 

Because the inhomogeneity is algebraically decaying, for large values of $r$ the relevant terms in the equation are
\[ \Delta_0 \phi - b(\partial_r \phi) ^2 + \Omega =0\]
To find a first order approximation, we can again use the Hopf-Cole transform, $\phi(r) = -(1/b) \log( K) $, rewriting the equation as
\[ \partial_{rr} K + \frac{1}{r} \partial_r K - \Lambda^2 K =0 \qquad \Lambda^2 = b \Omega. \]
Notice that this is either a Bessel, or the Modified Bessel equation, depending on the sign of $b \Omega$. Because we are interested in solutions, $\phi(r)$, that are real, we pick $b\Omega>0$ so that the solution to this last equation is $K=K_0$, the Modified Bessel function of the second kind. In particular, because $K_0(z) \sim \rmO(e^{-z})$ as $z \to \infty$
( see Table \ref{t:bessel} below), this implies that $\partial_r \phi(r) $ is bounded in the far field, as desired.

We therefore consider the ansatz 
\[\phi(r) = \phi_0(r) + \phi_1(r)\]
where $\phi_0$ is given by
\[ \phi_0(r) = -\frac{1}{b} \chi(\Lambda r) \log( K_0(\Lambda r)), \qquad \Lambda^2 = b \Omega >0.\]
Here, again $\chi$ represents a cut-off function that removes the singular behavior of the log function near the origin.
Inserting this expression into equation \eqref{e:far} gives
\[ \Delta_0 \phi_1 - 2b \partial_r \phi_0 \partial_r \phi_1 - b(\partial_r \phi_1)^2 
+ ( \Delta_0 \phi_0 - b(\partial_r \phi_0)^2 + \Omega) - \eps g =0.\]
Since the terms appearing in the parenthesis represent a localized function with compact support,
they do not contribute to the behavior of the solution for large values of $r$. Thus, the far field behavior of the solution is determined by 
\begin{equation}\label{e:far2}
 \Delta_0 \phi_1 - 2b \partial_r \phi_0 \partial_r \phi_1 - b(\partial_r \phi_1)^2 
 - \eps g =0.
 \end{equation}

Letting $\psi = \partial_r \phi_1$ and adding and subtracting the term $2 \Lambda \psi$ gives us
\[ \partial_r \psi + \frac{1}{r} \psi + 2 \Lambda \psi  + \left[ - 2b \partial_r \phi_0 \psi - b\psi^2 
- \eps g - 2 \Lambda \psi \right] =0.\]

We can then precondition this equation by $\mathcal{L}_{2 \Lambda}^{-1}$, since by Lemma \ref{l:Llambda} 
 we know that this operator is bounded for all values of $\Lambda>0 $, if its domain is $H^k_{r,\sigma}(\R^2)$.
 Thus, the equation can be written as
\begin{equation}\label{e:ift}
F(\psi; \eps) =  \mathrm{Id} +  \mathcal{L}_{2 \Lambda}^{-1}  \Big[ - 2b \partial_r \phi_0 \psi - b\psi^2 
 - \eps g - 2\Lambda \psi \Big] =0. 
\end{equation}

In what follows we will show that the operator 
$F: H^k_{r,\sigma}(\R^2) \times \R  \rightarrow H^{k}_{r, \sigma}(\R^2)$ 
satisfies the conditions of the implicit function theorem and prove the following theorem.

\begin{Theorem}\label{t:far}
Take $g \in H^k_{r,\sigma}$ with $k \geq 2$ and let $\sigma \in (0, 1)$. Then there exist a positive constant $\Lambda_0$ such that for any fixed $\Lambda \in (0,\Lambda_0)$, there is an
 $\eps_0>0$, and $C^1$  family of solutions, $\phi = \phi(r; \eps)$,
to equation \eqref{e:far2} that bifurcates from zero at $\eps =0$  and  is valid for $\eps \in (-\eps_0, \eps_0)  $. 
Moreover, this family has the form
\[ 
\phi(r;\eps) = -\frac{1}{b} \chi(\Lambda r) \log( K_0(\Lambda r)) + \phi_1(r; \eps) +\eps c  \]
where
\begin{enumerate}[i)]
\item $K_0(z)$ represents the zeroth-order Modified Bessel function of the second kind,
\item $ \partial_r \phi_1 \in H^{k+1}_{r,\sigma}(\R^2)$,
\item and $c \in \R$ is an arbitrary constant.
\end{enumerate}
\end{Theorem}

\begin{proof}
Because finding solutions to equation \eqref{e:far2} is equivalent to finding the zeros of the operator $F: H^k_{r,\sigma}(\R^2) \times \R  \rightarrow H^{k}_{r, \sigma}(\R^2) $ defined in 
\eqref{e:ift}, we check that $F$ satisfies the assumptions of the implicit function theorem.

It is clear that  $ F(0;0) =0$  and that this operator is smooth with respect to the parameter $\eps$. 
To check that the Fr\'echet derivative, $D_\psi F(0; 0): H^k_{r,\sigma}(\R^2) \rightarrow H^k_{r,\sigma}(\R^2)$,
given by
\[ D_\psi F(0;0) = Id + \mathcal{L}_{2\Lambda}^{-1} [ -2b \partial_r \phi_0 - 2\Lambda],\]
is invertible, notice that the term in the brackets can be written as the product of a bounded function times the constant $2\Lambda$. Indeed, this can be checked by expanding this term,
\begin{align*}
  -2b \partial_r \phi_0 -2 \Lambda & = -2b \partial_r[ \; - \frac{1}{b} \chi(\Lambda r) \log[ K_0(\Lambda r)] ] - 2\Lambda\\
&= 2 \Lambda \Big[ \chi'(\Lambda r)  \log[ K_0(\Lambda r)]  + \chi (\Lambda r) \frac{K_0'(\Lambda r) }{ K_0(\Lambda r)}  - 1\Big]
\end{align*}
and using the fact that the ratio $\frac{K'_0(z)}{K_0(z)} = -1$ as $r \to \infty$, and that $\chi'$ has compact support.
Since the operator $\mathcal{L}_{2\Lambda}^{-1} : H^{k}_{r,\sigma} (\R^2) \longrightarrow  H^{k+1}_{r,\sigma}(\R^2)$
is bounded, it follows that there is a small number $\Lambda_0>0$ such that if $\Lambda \in (0, \Lambda_0)$, the derivative  $D_\psi F(0;0)$ is a small perturbation of the identity
and is therefore invertible.

We are left with showing that the operator $F$ is well defined. 
Taking into account again that the map
$\mathcal{L}_{2\Lambda}^{-1} : H^{k-1}_{r,\sigma} (\R^2) \longrightarrow  H^{k}_{r,\sigma}(\R^2)$
is bounded, this is equivalent to showing that the terms
 \[  - 2b \partial_r \phi_0 \psi - b\psi^2 
 - \eps g - 2\Lambda \psi  \]
define a bounded operator $N: H^k_{r,\sigma}(\R^2) \times \R  \rightarrow H^{k-1}_{r, \sigma}(\R^2)$.

First, notice  that by assumption, the impurity $g $ is in the desired space.
As for the elements involving the variable $\psi \in H^k_{r,\sigma}(\R^2)$, because the derivative $\partial_r \phi_0$ is a bounded function, we can easily check that they are both in the space $H^{k-1}_{r,\sigma}(\R^2)$.
Finally, since $\sigma>0$, Lemma \ref{l:product} shows that the product $\psi^2$ is in $H^{k-1}_{r,\sigma}(\R^2)$.

 This proves that the operator $F$ satisfies the conditions of the implicit function theorem and proves the existence of a family of solutions solving $F(\psi;\eps) =0$ .
Going back to the definition of $\psi = \partial_r \phi_1$, we see that the above result
also gives us a family of solutions $\phi(r;\eps) = \phi_0(r) + \phi_1(r; \eps)+\eps c$ solving the far field equation \eqref{e:far}, where $\phi_1 \in H^{k+1}_{r,\sigma}(\R^2)$ and $c$ is for now an arbitrary constant which is the result of integrating $\psi$. This proves the result of the theorem.
\end{proof}

\section{Existence of Target Patterns}\label{s:existence}

\subsection{Matching}\label{ss:match}

To determine an expression for the eigenvalue $\Omega$,  we must match the intermediate and far field approximations of the wavenumber, $\partial_r \phi$. For convenience we recall their expressions,
\begin{align*}
\phi_{far}(r;\eps,\Lambda) = & -\frac{1}{b} \chi(\Lambda r) \log( K_0(\Lambda r)) + \phi_1(r; \eps) + \eps c  \\
\phi_{int}(r; \eps) = & -\frac{1}{b} \chi_M \log\Big( 1 + a  \log r + \eps K(r) \Big) +\bar{ \phi}_1(r;\eps) + \bar{\phi}_2(r;\eps)+ \eps c,
\end{align*}
As before, $K_0$ denotes the Modified Bessel function of the first kind, while the function $K$ satisfies
\[ \Delta_0 K + b g_f =0.\]
Notice that the remaining terms, $\phi_1,\bar{ \phi}_1,$ and $ \bar{\phi}_2$, all have derivatives that decay algebraically at infinity. In particular, 
\begin{enumerate}[1.]
\item The function $\bar{\phi}_1$ defined in Subsection \ref{ss:firstapprox} is in the space $R^k_\gamma \subset M^{2,2}_{r,\gamma-2}$, with $\gamma>1$. From Lemma \ref{l:decay} it follows that $|\bar{\phi}_1| < |x|^{-\gamma+1}$.
In particular, if the inhomogeneity $g_c \sim \rmO(r^{-(d +2)})$, with $d>0$, we have that $\partial_r \bar{\phi}_1 \sim \rmO(r^{-(d+1)})$.
\item From Theorems \ref{t:inter} and \ref{t:far} we know that the functions $\partial_r \phi_1$ and $\partial_r \bar{\phi}_2$ are in the space $H^k_\delta(\R^2)$, where  $ \delta\in (0,1)$. It then follows from Sobolev embeddings that these functions are bounded. In addition, because $\delta>0$, they must decay algebraically.
\end{enumerate}

\begin{table}[t]
\begin{center}
\begin{tabular}{ m{2cm} m{5cm} m{4.5cm}  } 
\specialrule{.1em}{.05em}{.05em} 
  & $z \to 0$ & $z \to \infty $\\
  \hline
 $K_0(z) $ &$ - \log(z/2) - \gamma_e + \rmO(z^2) $ &$ \sqrt{ \frac{\pi}{2 z} } \rme^{-z} \Big( 1 + \rmO(1/z) \Big)  $ \\ [3ex]
$ K_1(z) $ &$ \frac{1}{z} + \rmO(z)$ &$ \sqrt{ \frac{\pi}{2 z} } \rme^{-z} \Big( 1 + \rmO(1/z) \Big)  $\\      
\specialrule{.1em}{.05em}{.05em} 
\end{tabular}
\end{center}
\caption{ 
Asymptotic behavior for the Modified Bessel functions of the second kind of zeroth and first-order, taken from \cite[(9.6.8), (9.6.9), (9.7.2)]{abramowitz} }
\label{t:bessel}
\end{table}

To do the matching,
recall from the analysis in Subsection \ref{s:inter} that the parameter $\Lambda^2 = b \Omega$ 
 is assumed to be small beyond all orders of $\eps$. 
  This justifies the scaling $r = \eta(\eps) r_\eta$, where $r_\eta$ is a constant
 and $\eta(\eps) = \eps/ \Lambda $. 
  As a result, $\Lambda r \to 0$  as $\eps \to 0$, while $r \to \infty $, and we find that for small value of $\eta$ we are in the region where both approximations are valid. Moreover, since $ \eps \sim \rmo(\eta(\eps))$ there is always an open interval
  where the two approximations can be matched, even as $\eps \to 0$. 
  Because in this region  
  the functions $\chi_M= \chi =1$, we obtain
 \begin{align*}
\partial_r\phi_{far}(r;\eps,\Lambda) & \sim  -\frac{1}{b}\; \left[ \Lambda \frac{K'_0(\Lambda r)}{K_0(\Lambda r)} \right] + \partial_r \phi_1 \\
\partial_r \phi_{int}(r;\eps) & \sim - \frac{1}{b} \;  \left[
\frac{a/ r+ \eps \partial_r K(r) }{1 + a \log r + \eps K(r)} \right]  + \partial_r \bar{\phi}_1+ \partial_r \bar{\phi}_2.
\end{align*}
Setting the derivatives equal to each other,  $\partial_r \phi_{far} = \partial_r \phi_{int}$, we find that
\begin{align*}
\Big( \Lambda K'_0(\Lambda r) - b K_0(\Lambda r) \partial_r \phi_1 \Big) &
 \left( 1 +  a \log r + \eps K(r) \right) \\
& = K_0(\Lambda r)\Big [ (a/ r +  \eps \partial_r K(r)) - b ( \partial_r \bar{\phi}_1 + \partial_r \bar{\phi}_2)( 1 +  a \log r + \eps K(r) )\Big] \\
\Big( -\Lambda \left( \frac{1}{ \Lambda r} + \rmO(\Lambda r) \right) & -b K_0(\Lambda r) \partial_r \phi_1 \Big) 
 \left( 1 + a \log r + \eps K(r)  \right) \\
 &=\Big( - \log( \Lambda /2) - \gamma_e  - \log(r) + \rmO( (\Lambda r)^2) \Big)\\
 &\qquad \times \Big [ (a/ r + \eps \partial_r K(r)) - b( \partial_r \bar{\phi}_1 + \partial_r \bar{\phi}_2)( 1 +  a \log r + \eps K(r) )\Big], 
\end{align*}
where in the second line we use the fact that $K'_0(z) = -K_1(z)$ and the expansions from
Table \ref{t:bessel}. 

We now proceed with the matched asymptotic analysis to determine the value of 
$\Lambda$. Notice that due to the relation $\Omega b = \Lambda^2$, this will also allow us to
obtain  an expression for the frequency.
The method is as follows: We first divide the above expression by different gage functions 
in order to select terms of similar order in $\eps$. We then cancel any duplicate terms, 
 let $\eps$ go to zero, and
select the value of any undefined constant so that the remaining terms add up to zero.

Because we are interested only in finding the value of the constant $\Lambda$, 
we can simplify these computations by noticing that terms of the form $ \rmO(\Lambda r) \partial_r \phi_1$
will go to zero, as $\eps \to 0$, faster than any other  term.
Thus,  they are not of the same order in $\eps$ as elements that involve $\Lambda$.
This follows from the scalings picked and the algebraic decay rate of the the function $\partial_r \phi_1$.
 We may therefore consider instead the  expression
\begin{align}\label{e:match1}
 -\frac{1}{r} \left( 1 + a \log r + \eps K(r)  \right) = &(- \log( \Lambda /2)  - \gamma_e  - \log(r) )\\
 \nonumber
 & \times  [ (a/ r + \eps \partial_r K(r)) - b( \partial_r \bar{\phi}_1 + \partial_r \bar{\phi}_2)( 1 +  a \log r + \eps K(r) )].
\end{align}
It is worth pointing out here that, in contrast to more standard matched asymptotic analyses,
the elements in equation \eqref{e:match1} are not of order $\rmO(\eps^n), n \in \N$. Moreover, we find that dominants terms depend on
the yet to be determined approximations $\bar{\phi}_i, i =1,2$. Thus, we will 
not be able to match them exactly, but we can justify that the process can be done.

First, looking at the right hand side, one notices that the dominant term is $ - \eps \log(r) \partial_r K$.
Because $r$ depends on $\eps$, we may use this function as a gage function. 
Dividing by $ - \eps \log(r) \partial_r K$ and letting $\eps \to 0$, or equivalently $r \to \infty$, 
we are left with matching,
\[ 0 =  1 - b \frac{ \partial_r \bar{ \phi}_1+ \partial_r \bar{ \phi}_2 }{ \eps \partial_r K} ( 1 + a \log r + \eps K). \]
By picking the value of $\delta \in (0,1)$ so that the higher order correction term, $\bar{\phi}_2 \in L^2_\delta(\R^2)$, is in the same
space as both, $K$ and $ \partial_r K$, we see that it is possible to match these terms. 
Expression \eqref{e:match1} then becomes
\begin{align*}
 -\frac{1}{r} \left( 1 + a \log r + \eps K(r)  \right) = &(- \log( \Lambda /2)  - \gamma_e  )   \\
 & \times [ (a/ r + \eps \partial_r K(r))- b( \partial_r \bar{\phi}_1 + \partial_r \bar{\phi}_2)( 1 +  a \log r + \eps K(r) )] \\
 &- \frac{a}{r} \log r.
\end{align*}

Second, cancelling the term $\frac{a}{r} \log r$, using $ \eps K/r$ as a gage function, and letting $r \to \infty$,
we obtain
\[ -1 = (- \log( \Lambda /2)  - \gamma_e  ) \left[ \frac{ \partial_r K}{r K } - \frac{b r}{ \eps K} ( \partial_r \bar{\phi}_1 + \partial_r \bar{\phi}_2)( 1 +  a \log r + \eps K(r) )\right ]. \]
Since $\bar{\phi}_i, i =1,2$ represent all higher order terms, and not just one function, one can again justify  that these terms can be matched.
As a result, equation \eqref{e:match1} now reads
\[ -\frac{1}{r} = ( - \log \Lambda/2 - \gamma_3) \frac{a}{r}.\]

Finally, solving for $\Lambda$, we see that $\Lambda=  2 \rme^{-\gamma_e} \exp(1/a)$. 
Using the relation $\Lambda^2 = b \Omega$, we also obtain that
\[\Omega =  \frac{4}{b}  \rme^{-2\gamma_e} \exp\left( \frac{1}{ a } \right).\]
In particular, from the definition of $a$, i.e. $a = -b \eps \int_0^\infty g_c(r) \;r\;dr$, we may conclude that both $\Lambda$ and 
$\Omega$ are smooth functions of $\eps$, for all $\eps \in (0, \eps_M)$, with $\eps_M$ a positive constant.
In addition, notice that as $\eps$ approaches zero, 
the value of $\Lambda$ and $\Omega$ also goes to zero.

\begin{Remark}\label{r:Omega}
Notice that:
\begin{enumerate}[1.]
\item We need the constant  $a<0$ in order for $\Lambda=b\Omega$ to satisfy our initial assumption of being small beyond all 
orders of $\eps$. If $\eps>0$, this condition is guaranteed from formula \eqref{e:constantA} and the assumption that $g$ is a positive function.
\item Notice also that if  $ \eps \int g <0$, the gradient $\partial_r \phi_{int}$ would 
also be negative and we would not be able to match the two approximations. 
This is in line with previous results which show that target pattern solutions (or thanks to the Hopf-Cole transform, $\phi = -\frac{1}{b} \log( \Psi) $, ground states of the Schr{\"o}dinger eigenvalue problem, $\Delta \Psi + \eps g \Psi$) do not exist when the inhomogeneity (potential) satisfies $\eps \int g <0$. See \cite{simon1979} for a proof of this result.

\item  Because we rigorously proved the existence of solutions to the intermediate and far field approximations,
we know that we can obtain approximations for $\phi_{far}$ and $\phi_{inter}$ to any desired order. 
Thus, by matching these higher order approximations, we can obtain better estimates for the
parameter $\Lambda$. In particular,
if we consider $a  = \eps a_1 + \eps^2 a_2 $, and find the corresponding expressions for $\phi_{far}, \phi_{inter}$ and $ a_2$,
the above matching process leads to $\Lambda = C(\eps)  2 \rme^{-\gamma_e} \exp(1/\eps a_1)$,
with $C(\eps) = \exp(1/a - 1/\eps a_1) $.
 In addition,
by defining $\Lambda(0) = \partial_\eps \Lambda(0) =0$, we  obtain that this estimate is also  $C^1$ with respect to $\eps$ on $[0,\eps_M)$, for some $\eps_M>0$.
\end{enumerate}
\end{Remark}

\subsection{Existence of Solutions}\label{ss:existence}
In this subsection we combine the results of the previous subsections and prove Theorem \ref{t:existence}, which is stated in the introduction and reproduced below for convenience.

\begin{Theorem*}
Let $k \geq 2$ and $\sigma \in (0, 1)$ and consider a function $g \in H^k_{r,\sigma}(\R^2)$ satisfying Hypothesis \ref{h:g}. Then, there exists a constant $\eps_0>0$ and a $C^1 ([0,\eps_0))$  family of eigenfunctions $\phi = \phi(r; \eps)$ and eigenvalues $\Omega= \Omega(\eps)$ that bifurcate from zero and solve
 the equation
\begin{equation}\label{e:far3}
\Delta_0 \phi - b(\partial_r \phi) ^2 - \eps g(r) + \Omega =0 \qquad r= |x| \in [0,\infty).
\end{equation}
Moreover, this family has the form
\[ 
\phi(r;\eps) = -\frac{1}{b} \chi(\Lambda r) \log( K_0(\Lambda r)) + \phi_1(r; \eps ) + \eps c , \qquad \Lambda^2 = b \Omega(\eps) \]
where
\begin{enumerate}[i)]
\item $c$ is a constant that depends on the initial conditions of the problem,
\item $K_0(z)$ represents the zeroth-order Modified Bessel function of the second kind,
\item $ \partial_r \phi_1 \in H^{k}_{r,\sigma}(\R^2) $, and
\item $\Omega= \Omega(\eps) = C(\eps) 4 e^{-2 \gamma_\eps}\exp[ 2/ a ]$, with
\[a = -\eps b \int_0^\infty g_c(r) \;r \;dr,\]
and $C(\eps)$ a $C^1$ constant that depends on $\eps$.
\end{enumerate}
\end{Theorem*}

\begin{proof}
The proof mimics the analysis done for the far field approximation, except that now we consider the full equation
\eqref{e:far3}. As above, we use the ansatz $\phi(r) = \phi_0(r) + \phi_1(r)$, with
 $\phi_0$ given by
\[ \phi_0(r) = -\frac{1}{b} \chi(\Lambda r) \log( K_0(\Lambda r)), \qquad \Lambda^2 = b \Omega >0.\]
In contrast to the analysis from Section \ref{s:far}, here we treat the parameter
$\Lambda$ as a $C^1$ function of $\eps$, 
a result that follows from the matched asymptotic analysis of 
Subsection \ref{ss:match}.
Thus, given any $\eps>0$ there is a corresponding value of $\Lambda$ 
that defines an approximation, $\phi_0$, and a frequency, $\Omega = \Lambda^2/b$, both of which satisfy the equation
 $( \Delta_0 \phi_0 - b(\partial_r \phi_0)^2 + \Omega) =0$ in the far field.
 
Inserting this ansatz into equation \eqref{e:far3} gives
\[ \Delta_0 \phi_1 - 2b \partial_r \phi_0 \partial_r \phi_1 - b(\partial_r \phi_1)^2 
+ ( \Delta_0 \phi_0 - b(\partial_r \phi_0)^2 + \Omega) - \eps g =0.\]
Letting $\psi = \partial_r \phi_1$, adding and subtracting the term $2 \Lambda \psi$,
and precondition the result by $\mathcal{L}_{2 \Lambda}^{-1}$, gives the following equivalent
formulation of equation \eqref{e:far3},
\begin{equation}\label{e:ift2}
F(\psi; \eps) =  \mathrm{Id} +  \mathcal{L}_{2 \Lambda}^{-1}  \Big[ - 2b \partial_r \phi_0 \psi - b\psi^2 
+ ( \Delta_0 \phi_0 - b(\partial_r \phi_0)^2 + \Omega) - \eps g - 2\Lambda \psi \Big] =0. 
\end{equation}
Our goal is to show that the operator 
$F: H^k_{r,\sigma}(\R^2) \times \R  \rightarrow H^{k}_{r, \sigma}(\R^2)$ 
satisfies the conditions of the implicit function theorem.

By Remark \ref{r:Omega}, 
$\Lambda(0) =\partial_\eps \Lambda(0) =0$, so that the operator $F$ is 
$C^1([0,\eps_M))$ with respect to $\eps$, for some $\eps_M>0$. Moreover, thanks to the cut-off function 
 in the definition of $\phi_0$, i.e. $\chi= \chi(\Lambda r)$,
we find that the terms 
$ ( \Delta_0 \phi_0 - b(\partial_r \phi_0)^2 + \Omega)$ tend to zero as $\eps$ goes to zero.
Therefore, $F(0;0) = 0$.
In addition, because the elements in the parenthesis are smooth and have compact support, they belong to the space $H^{k-1}_{r, \sigma}(\R^2)$, for any  natural number $k$ and any real number $\sigma$. 
A similar analysis as in the proof of Theorem \ref{t:far} then shows that the rest of 
the terms in $F$ belong to the space $H^k_{r, \sigma}(\R^2) $, with $k \geq 2,$ and $ \sigma \in (0,1)$.
As a result, the operator $F$ is also well defined.
Since
 its Fr\'echet derivative,  $D_\psi F(0;0)$, is now the identity map on $H^k_{r,\sigma}(\R^2)$,
we may apply the implicit function theorem 
to conclude the existence of solutions $\psi = \partial_r \phi \in H^k_{r,\sigma}(\R^2)$. The results of Theorem \ref{t:existence} 
then follow in a similar way as those done in Section \ref{s:far}. 
\end{proof}

\section{Simulations}\label{s:simulations}

In this section we numerically explore the effects of adding large inhomogeneities, $\eps g$, as perturbations to the eikonal equation, i.e.
\begin{equation}\label{e:eiknum}
\phi_t = \Delta \phi - | \nabla \phi |^2 - \eps g.
\end{equation}
To run the simulations we model the equation on a square domain   with periodic boundary conditions and employ a spectral RK4 method based on \cite{kassam},  using a mesh size $h=100/512$  and a time step $dt =0.5$. The numerical scheme is continued until a steady state is reached.  Different domain lengths were tested, ($L = \{100,120,140, 160,180, 200\}$), resulting in the same approximations for $\phi$. Thus, a domain of length $L =100$ was chosen to run all numerical experiments for computational efficiency.

  \begin{figure}[t] 
     \centering
     \begin{subfigure}{0.4\textwidth}
     \includegraphics[width=\textwidth]{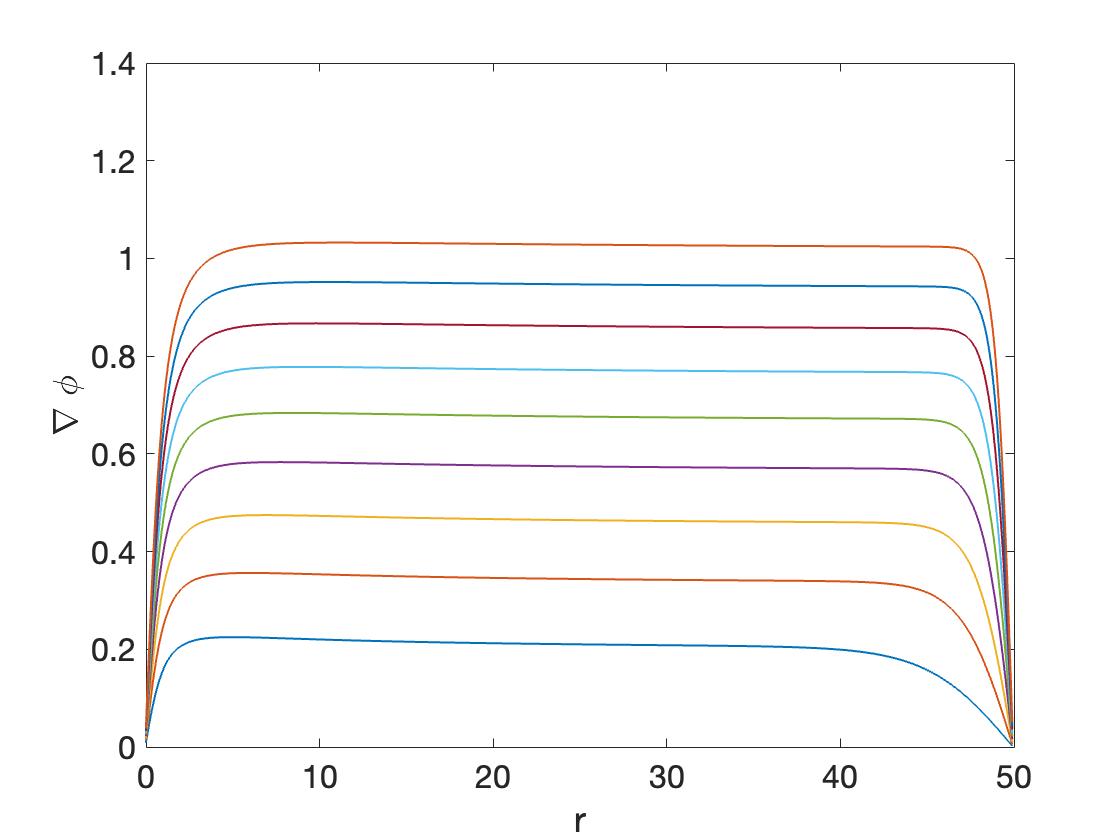} 
     \caption{}
     \label{f:waveA}
     \end{subfigure}
     \begin{subfigure}{0.4\textwidth}
     \includegraphics[width=\textwidth]{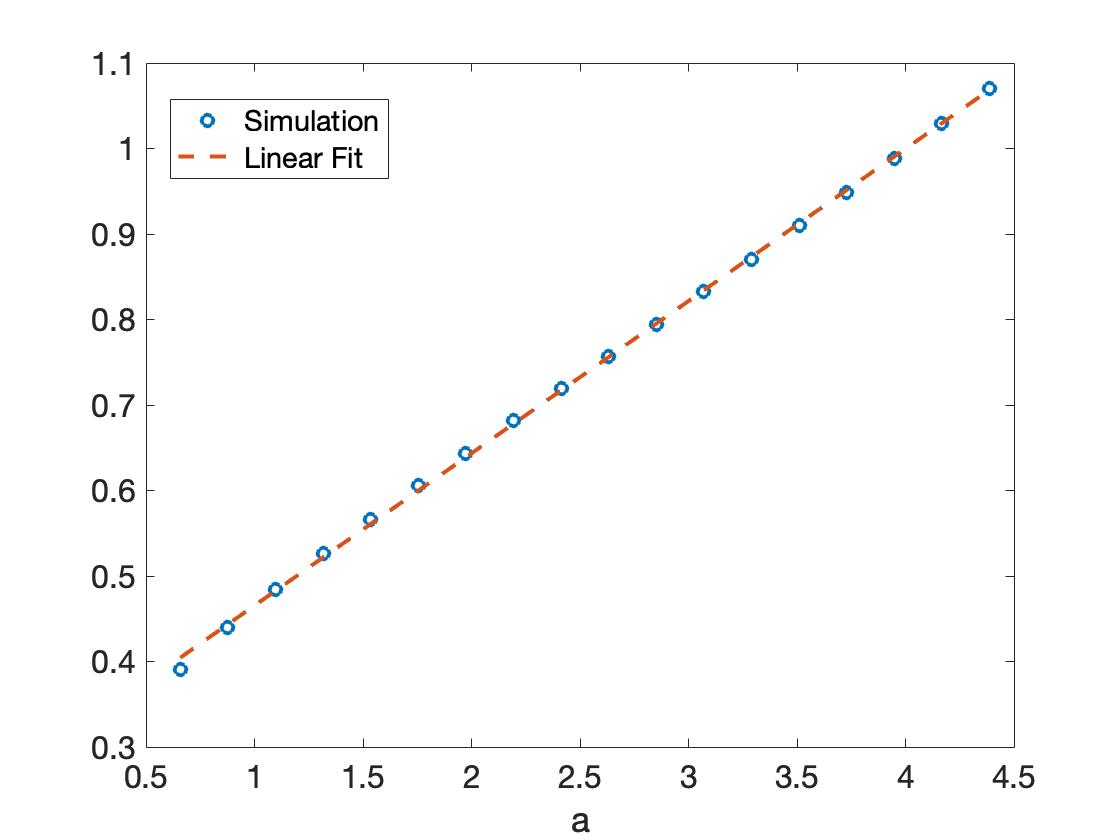} 
     \caption{}
     \label{f:waveB}
     \end{subfigure}
     \caption{ Numerical simulation of the time dependent eikonal equation \eqref{e:eiknum} with $g$ as in \eqref{e:g1g2}, $p =0.8$, initial condition $\phi=0$,  and various values of $\eps \sim a =0.15 * (3:2:20)$, where $a =\eps \int_0^3 g(r) r \;dr$. A) Plots of the gradient of the steady state solution. Top most curve corresponds to maximum value of $\eps$ used. 
     B) Plot of $\frac{1}{\log(k(a)) -1}$ vs. $a$, where for large $|x|$ the gradient $\nabla \phi$ approximates the wavenumber, $k(a)$.Circles represent data from the simulation while dotted line is the linear fit.}
     \label{f:wave}
  \end{figure}

Simulations confirm our analytical results,  finding that for inhomogeneities 
that take  the form
 \begin{equation}\label{e:g1g2}
  g = \left(  \frac{A}{( 1 +r^2)^{p}} \right), \quad p \in (1/2, 1],\quad A \in \R,
  \end{equation}
 the solutions to the eikonal equation grow linearly at infinity.  
  This is depicted in Figure \ref{f:waveA} where the gradient, $\nabla \phi$, is plotted 
 for different values of the parameter $\eps$. Notice that because we are using periodic boundary conditions,
 the value of $\nabla \phi$ goes to zero at the boundary of the domain.
  As predicted by the analysis of the previous sections, 
  we find that  the wavenumber, $k = \lim_{|x| \to \infty} \nabla \phi  \sim \Lambda/b $, and as a result 
  the frequency, $\Omega = \Lambda^2/b$, is small beyond all orders of $\eps$. 
  To confirm this result we approximate the wavenumber
  by evaluating the gradient $\nabla \phi$ at large values of $|x|$.
 In  Figure \ref{f:waveB} we plot the relation $\frac{-1}{\log(k) -1}$ vs. $a$,
 where $a$ represents the mass  of $g_c = (1- \chi)g$, 
 which we take as a substitute for $\eps$, since $a = -b \eps \int g_c(r) r \;dr$. 
  Notice how in the figure  the data points taken from the simulations follow a straight line, confirming  that $\Lambda \sim \exp(1/a)$.

Finally, to determine how the the decay rate, $p$, affects the wavenumber, we ran simulations for values of $p \in (0.5, 3]$. Notice that using the notation from Hypothesis \ref{h:g}, where $g \sim 1/r^m$, this is equivalent to considering values of $ m \in (1, 6)$.
These results are summarized in Figure \ref{f:powerA}.
They show that the wavenumber decreases as the decay rate of the inhomogeneity, $p$, increases. 
The figure also compares the numerical approximation to the wavenumber, $k$, 
which we plot using stars, with the analytical result $k  \sim \exp(1/a)$. In particular, following 
 Theorem \ref{t:existence} we use
 
\[ a = \left \{ \begin{array}{c c c }
-\eps b \int_0^3 g(r) \;r\;dr & \mbox{for} & p \in (1/2,1)    \\
-\eps b \int_0^\infty g(r) \;r\;dr & \mbox{for} & p \in (1,3) .
\end{array}
\right.\]
For  values of $p \in (0.5,1) \sim m \in (1,2)$, we are in the regime considered in this paper, 
where the impurity $g$ does not have finite mass
and is thus a large inhomogeneity. In this case, we assume a value of $D \sim 3$ in the definition of 
$g_c$ specified in the introduction, see equation \eqref{e:g}.
  We then calculate the mass of this function 
by integrating from 0 to 3, since this provided the best fit to the data (see Remark \ref{r:D}).
This approximation is plotted using a dashed line.
On the other hand, when $p \in (1,3)$ we are in the regime where the impurity, $g$, has finite mass and
the results from \cite{jaramillo2018} apply (see also Theorem 1 and Remark \ref{r:impurity}).
In this case, the value of $a = - \eps b \int_0^\infty g(r) \;r\;dr$.
This approximation is plotted using  a solid line. 

Notice that both approximations for the wavenumberm, $k$, do a good job of following the data in the respective regions of the $p$-axis 
where they are valid, i.e. $0.5< p < 1$ for the dashed line, $p>2$ for the solid line. However, the estimates 
for $p \in (1,2)$ using the mass of $g$ (solid line) are not accurate, even though they follow the results from Theorem 1 in \cite{jaramillo2018}, or equivalently,
Theorem 1 together with Remark \ref{r:impurity} stated in this paper. This is not unreasonable given that the frequency of the pattern, $\Omega$,
and as a result its wavenumber, $k$, are both small beyond all orders of the parameter $\eps$. 
In particular,
when $p \to 1$ we have that $a = -\eps b \int g(r) r\;dr \to -\infty$.
Because $a \sim \rmO(\eps)$, the estimates for $\Omega \sim \exp(1/a) $
become worse and worse, and in this case one needs to approximate $a$ to higher orders in $\eps$ to obtain better estimates. 
Figure \ref{f:powerA} then suggests that the interval $1<p<2$ is a transitional 
regime, where one can numerically obtain a better fit to the data by using a cut-off function to better
approximate the value of $a$.

Finally, we also confirm numerically that for values of $p\leq 0.5$ the 
inhomogeneity no longer produces target patterns, but rather solutions with $\nabla \phi \sim \rmO(r)$ at infinity, see Figure \ref{f:powerB}. This is not a tight bound on the growth rate of $\nabla \phi$ and is just a very rough estimate based on our numerical experiments. 

 \begin{figure}[t] 
   \centering
    \begin{subfigure}{0.4\textwidth}
   \includegraphics[width=\textwidth]{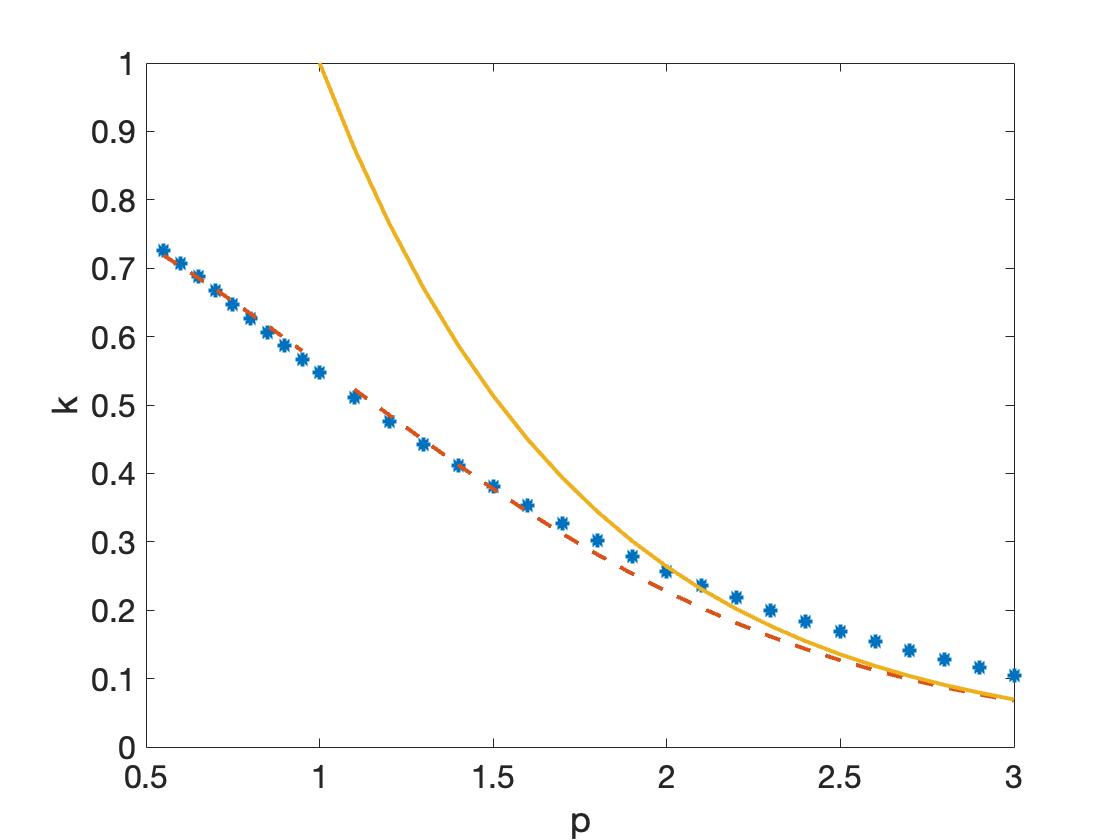} 
   \caption{}
   \label{f:powerA}
   \end{subfigure}
    \begin{subfigure}{0.4\textwidth}
   \includegraphics[width=\textwidth]{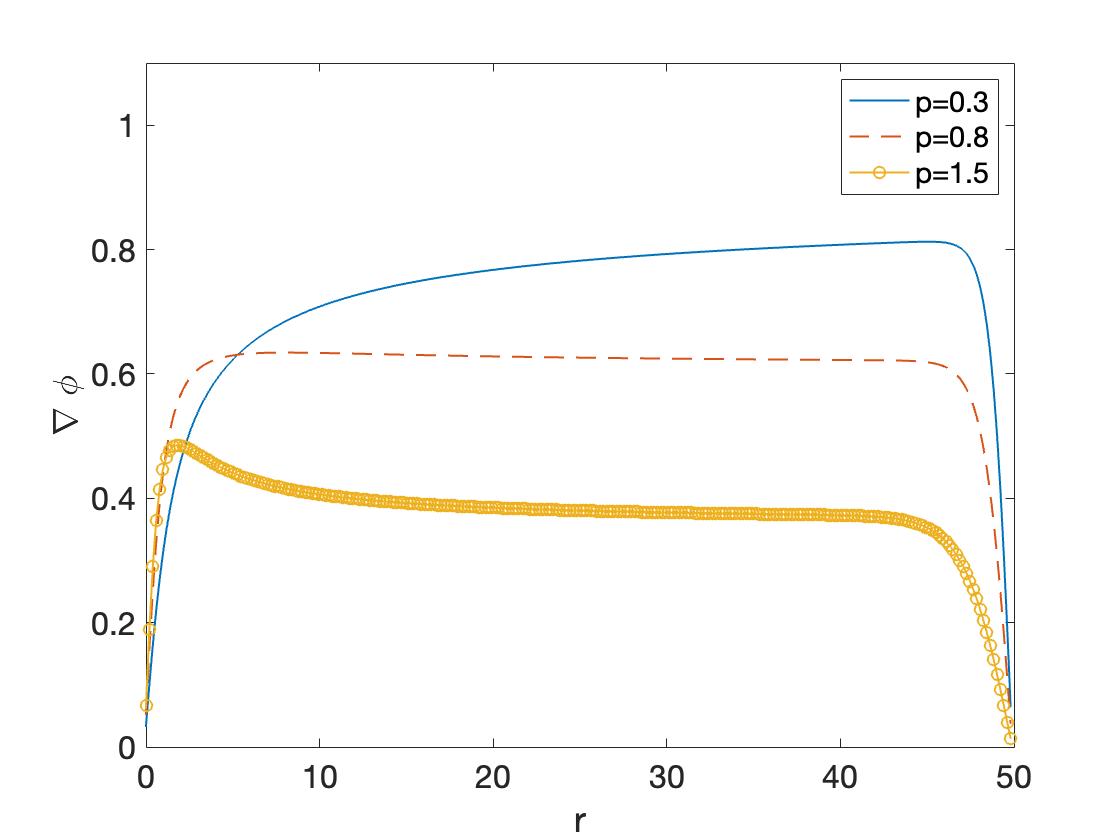} 
   \caption{}
   \label{f:powerB}
  \end{subfigure}
   \caption{\small A) Plot of wavenumber $k$ vs. $p$ for steady state solutions of the eikonal equation using $g=\dfrac{A}{(1+r^2)^p}$ with $A = 1.5$ and  for $ p \in (0.5, 3)$. Stars represents results from simulation, while solid and dashed lines represents approximation with $k \sim \exp(-1/a)$,  with $a= -A/(2-2p)$ for solid line, $a = A \int_0^3 g(r) \;rdr $ for dashed line.
   B) Plot of $\nabla \phi$ vs. $r$, for values of $p=0.3,0.8,1.5$. }
   \label{f:power}
\end{figure}

\section{Discussion}\label{s:discussion}

In this paper we showed that large defects can generate target patterns in oscillatory media.
Under the assumption of weak coupling, we modeled such systems using
a viscous eikonal equation, and represented the defect as a localized inhomogeneity.
In contrast to previous results, which assume that the inhomogeneity is strongly localized,
in this paper we relaxed this assumption and described impurities as functions with algebraic
decay of order $\rmO(1 /|x|^m)$, $1< m \leq 2$.

Our main motivation for studying this problem came from 
the universality of the viscous eikonal equation as a model for the phase dynamics 
of coherent structures in oscillatory media.
 In particular, our interest stems from the fact that this same equation can be 
used to describe the phase dynamics of spiral waves in oscillatory media with nonlocal coupling.
In this context, the large inhomogeneity no longer represents a defect, but instead encodes information about
variations in the amplitude of the pattern. 

A second motivation came from the fact that the steady state viscous eikonal equation is conjugate
to a Schr\"odinger eigenvalue problem. 
Indeed, it is well known that the Hopf-Cole transformation maps target pattern solutions
to bound states of the corresponding Schr\"odinger operator,
and that the frequency of target pattern solutions then corresponds to the energy of these states. 
In this context, the results presented here expand the conditions on the
Schr\"odinger potential that allow for such bound states to exist.
In particular, we show that Schr\"odinger operators with potentials that decay sufficiently fast  at infinity 
can have bound states even when the mass  of the potential $\int_{\R^2} g(r) \;r\;dr$ 
is not finite.

In particular, our analysis provides a first order approximation for target pattern solutions and
for their frequency.
In agreement with simulations  we show that, just as in the case of small defects,
the frequency is small beyond all orders of the small parameter
used to describe the strength of the impurity.
As a result, solutions do not follow a regular expansion. 
Therefore, to obtain our results we first found intermediate and far field approximations
 to the steady state viscous eikonal equation.
Then using a matched asymptotic analysis we were able to determine the value of the frequency selected
by the system. 
This approach is similar in spirit to the one used to prove existence of target patterns and spiral
waves in reaction-diffusion equations using spatial dynamics, \cite{scheel1998, kollar2007}.
There, the modeling equations are viewed as a system of ordinary differential equation in the radial variable, and
a center manifold reduction is used to obtain a vector field describing the amplitude
of these patterns. Coherent structures then correspond to heteroclinic solutions, connecting
a fixed point at infinity with solutions that are bounded near the origin.
Our matching process is then equivalent to showing that the center-stable
manifold of the fixed point intersects transversely the solution curve 
that lives in the center manifold.

Finally, the analysis presented in this paper is  complemented by simulations of the viscous eikonal
equation. Our numerical experiments are in good agreement with simulations.
They confirm that the wavenumber, and therefore the frequency of target patterns, do not follow a 
regular expansion on the small parameter $\eps$ representing the strength of the impurity $g$.
They also confirm that when $m \leq 1$, the solutions to the 
viscous eikonal equation no longer represent target patterns, since in this case
the gradient $\nabla \phi$ does not approach a constant as $|x| \to \infty$.

\section{Appendix}\label{a:A}

In \cite{jaramillo2022} it was shown that the following amplitude equation governs the dynamics of one-armed spiral waves in nonlocal oscillatory media,
\[ 0 = \beta \Delta_1 w + \lambda w + \alpha |w|^2w + N(w,\eps), \quad r\in [0,\infty).\]
Here $w$ is a radial and complex-valued function, and
\[  \beta = ( \sigma - \eps \lambda), \quad \lambda, \alpha \in \C,\quad  N \sim \rmO(|\eps||w|^4). \]
It was also established in \cite{jaramillo2022} that the constant $\lambda_I$ is an unknown parameter that needs to be determined when solving the equation.

In this section a multiple-scale analysis is used to derive a steady state viscous eikonal equation from the above
expression. We will see that this eikonal equation is of the form considered in this paper and that it involves an inhomogeneity
that decays at order $\rmO(1/|x|^2)$.

 To accomplish this task we first let $w = A \tilde{w}$, with $A^2 =- \lambda_R/\alpha_R$. This change of variables is done for convenience and leads to the following equation,
\[ 0 = \beta \Delta_1 \tilde{w} + \lambda \tilde{w} + (-\lambda_R + \rmi \tilde{\alpha}_I) |\tilde{w}|^2\tilde{w} + N(\tilde{w},\eps), \qquad \tilde{\alpha}_I = - \alpha \lambda_R/\alpha_R.\]
Letting $\tilde{w} = \rho \rme^{\rmi \phi}$ and separating the real and imaginary parts of the above expression, one finally
 obtains the system
\begin{eqnarray}\label{e:real}
0 = & \beta_R \left[ \Delta_1 \rho - (\partial_r \phi)^2 \rho \right] - \beta_I \left[ \Delta_0 \phi \rho + 2\partial_r \phi \partial_r \rho \right] + \lambda_R \rho - \lambda_R \rho^3 + \mathrm{Re}\left[ N(\tilde{w}; \eps) \rme^{-\rmi \phi} \right]\\ \label{e:imag}
0 = &  \beta_R \left[ \Delta_0 \phi \rho + 2\partial_r \phi \partial_r \rho \right] + \beta_I \left[ \Delta_1 \rho - (\partial \phi)^2 \rho \right] + \lambda_I \rho + \tilde{\alpha}_I \rho^3 + \mathrm{Im}\left[ N(\tilde{w}; \eps) \rme^{-\rmi \phi} \right].
\end{eqnarray}

Next, we proceed with a perturbation analysis following \cite{doelman2009}.
We rescale the variable $r$ by defining $S = \delta r$, where $\delta$ is assumed to be a small positive parameter. We also use the following expressions for the unknown functions:
\[\begin{array}{r l c l }
\rho = & \rho_0 + \delta^2 (R_0 + \delta R_1), &   \rho_0= \rho_0(r),&  R_i = R_i(\delta r) \quad i=0,1\\
\phi = & \phi_0 + \delta \phi_1, & & \phi_i = \phi_i(\delta r) \hspace{0.5cm}  i =0,1.
\end{array}\]
And for the parameter we choose
$ \lambda_I = -\tilde{\alpha}_I + \delta^2 \tilde{\lambda}_I,$ with $\tilde{\alpha}$ as above and $\tilde{\lambda}_I$ a 
free parameter.

Inserting the above ansatz into the equations \eqref{e:real} and \eqref{e:imag} we obtain a set of equations in powers of
$\delta$. To write this equations more compactly, we use the subscript $S$ to distinguish operators that are applied to functions that depend on this variable, i.e. $\Delta_{0,S}$. The absence of this subscript indicates that the operator is applied to a function of the original variable $r$.

At order $\rmO(1)$ we find that $\rho_0$ must satisfy,
\begin{align*}
0 = & \beta_R \Delta_1\rho_0 + \lambda_R \rho_0 - \lambda_R \rho_0^3,\\
0 = & \beta_I \Delta_1 \rho_0 - \tilde{\alpha}_I \rho_0 + \tilde{\alpha}_I \rho_0^3.
\end{align*}

At the next order, $\rmO(\delta^2)$, we find two equations involving $R_0$ and $\phi_0$,
\begin{align*}
0 = & - \beta_I \rho_0 \Delta_{0,S} \phi_0  - 2\beta_I \partial_S \phi_0 \partial_S \rho_0  - \beta_R \rho_0 (\partial_S \phi_0)^2   + \lambda_R R_0( 1- 3 \rho_0^2),\\
0 = & \beta_R \rho_0 \Delta_{0,S} \phi_0 + 2\beta_R \partial_S \phi_0 \partial_S \rho_0  -\beta_I \rho_0 (\partial_S \phi_0)^2 +
\tilde{\alpha}_I R_0 ( 3 \rho_0^2-1) + \tilde{\lambda}_I \rho_0 .
\end{align*}
For our purposes, it is enough to stop at this stage and not list higher order terms.

We first focus on the order $\rmO(1)$ system. The first equation can be solved, provided $\beta_R, \lambda_R>0$.
This equation falls into a broader family of o.d.e. which were solved in \cite{Kopell1981spiral}.
In this reference, the authors showed that such equations posses a unique solution $\rho_*$ satisfying
\[ \rho_* \to 1 \quad \mbox{as} \quad r\to \infty, \qquad \rho_*(r) \sim br \quad \mbox{when} \quad r \sim 0\]
Of course, the solution $\rho_*$ would not satisfy the second equation in the system. So we let
\[ G = \beta_I \Delta_1 \rho_* - \tilde{\alpha}_I \rho_* + \tilde{\alpha}_I \rho_*^3 
= \left(\frac{\beta_I}{\beta_R} \lambda_R + \tilde{\alpha}_I \right) \rho_* ( \rho_*^2-1),\]
and add these terms to the order $\rmO(\delta^2)$ system.

\begin{figure}[t] 
   \centering
   \includegraphics[width=3in]{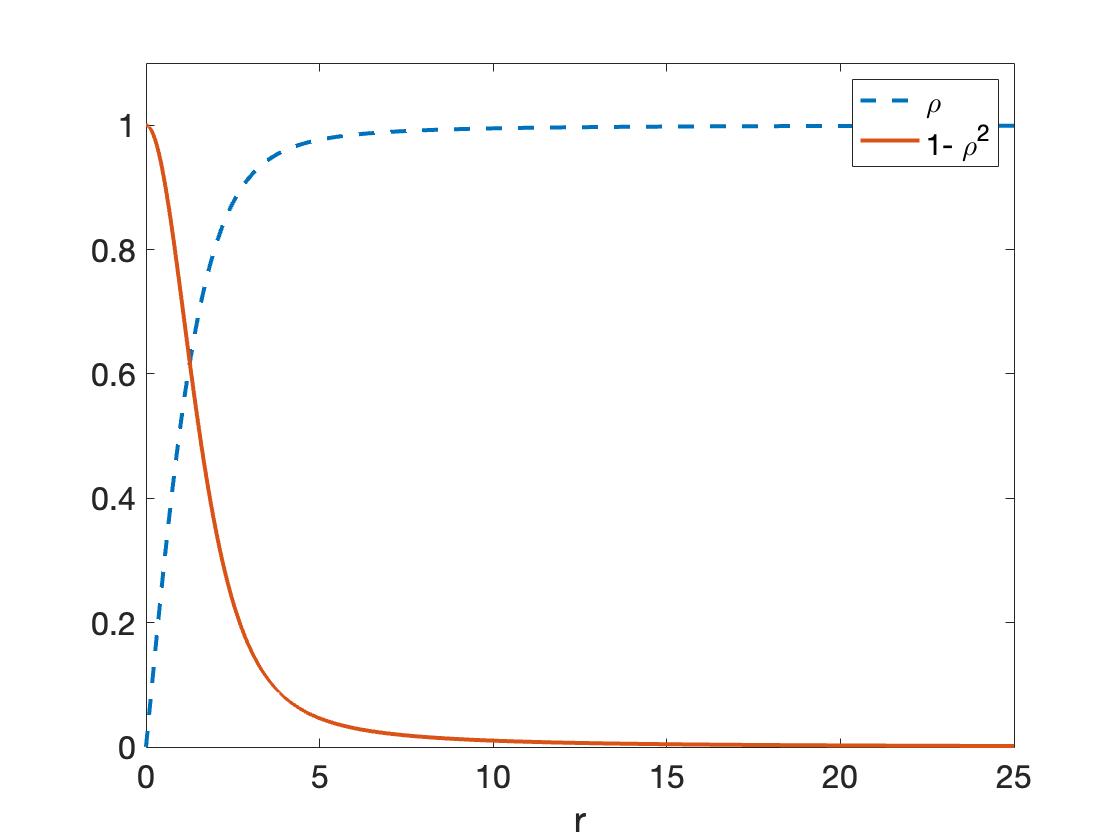} 
   \includegraphics[width=3in]{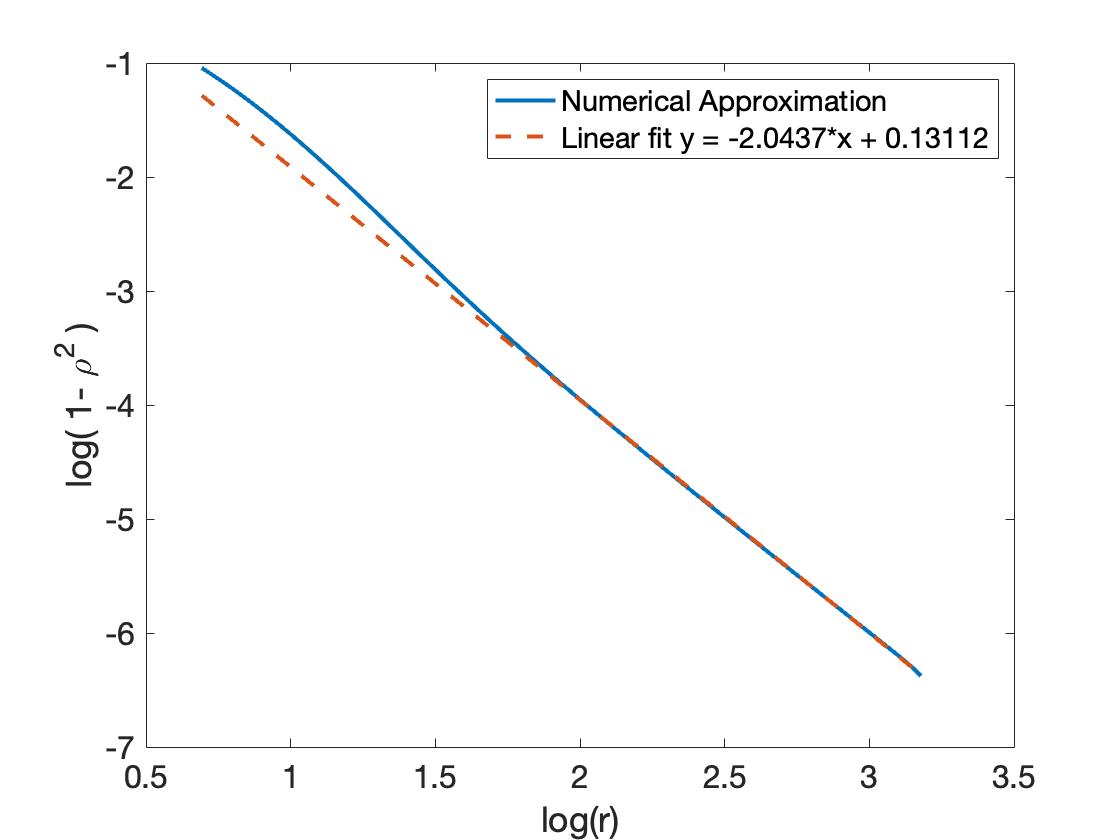} 
   \caption{Solution to the boundary value problem \eqref{e:bvp}}
   \label{f:rho}
\end{figure}

Going back to the order $\rmO(\delta^2)$ system, we first notice that because $\rho_0 =\rho_* \sim br = bS/\delta$ near the origin,
then the terms that involve this variable are in fact 'large' when compared to the terms that do not. 
Concentrating only on these large terms, we find that in the first equation we can solve for $R_0$ in 
terms of the variable $\phi_0$. Inserting this result into the second equation gives us
the viscous eikonal equation, 
\[ \Delta_{0,S} \phi_0 - b (\partial_S \phi_0)^2 +\Omega - c g=0\]
as expected, where
\[ b = \frac{\beta_I \lambda_R - \beta_R \tilde{\alpha}_I}{\tilde{\alpha}_I \beta_I + \lambda_R \beta_R},\qquad
\Omega = \frac{\tilde{\lambda}_I\lambda_R}{\tilde{\alpha}_I \beta_I + \lambda_R \beta_R}, \qquad
 c = -\frac{ \beta_I \lambda_R + \tilde{\alpha}_I \beta_R}{ \beta_R( \tilde{\alpha}_I \beta_I + \lambda_R \beta_R)},\qquad
 g = (1 - \rho_*^2 ).\]

Numerical simulations show that the perturbation $g$ decays at order $\rmO(1/r^2)$ as $r$ goes to infinity, see Figure \ref{f:rho}. To obtain these results, we solved the boundary value problem 
\begin{equation}\label{e:bvp}
0 =  \partial_{rr} \rho + \frac{1}{r} \partial_r \rho - \frac{1}{r^2} \rho +  \rho - \rho^3,\qquad \rho(\infty) = 1, \quad \rho(0) =0,
\end{equation}
treating the equation as a system of o.d.e. and using a shooting method with condition
\[ \rho(r) \sim br \quad \mbox{when} \quad r \sim 0.\]

\section{Declarations}
{\bf Conflict of Interest:} The author declares that she has no conflict of interest.
\bibliographystyle{plain}
\bibliography{spirals}

\end{document}